\documentclass[matprg]{svjour3}
\smartqed  
\usepackage{graphicx}
\usepackage{amsmath,amssymb,amstext}
\usepackage{mathrsfs}
\usepackage{multirow}
\usepackage{threeparttable}
\usepackage[multiple]{footmisc}
\usepackage{algorithm}
\usepackage{algorithmic}
\usepackage{subfigure}
\usepackage{color}
\usepackage{cases}
\usepackage{subfig}
\usepackage{booktabs}
\usepackage{diagbox}

\usepackage{xcolor}
\usepackage{enumitem}

\makeatletter
\def\@thmcountersep{.}
\makeatother
\spnewtheorem{defi}{Definition}[section]{\bfseries}{\normalfont}
\spnewtheorem{alg}{Algorithm}[section]{\bfseries}{\normalfont}
\spnewtheorem{ass}{Assumption}[section]{\bfseries}{\normalfont}
\spnewtheorem{lem}{Lemma}[section]{\bfseries}{\normalfont}
\spnewtheorem{thm}{Theorem}[section]{\bfseries}{\normalfont}
\spnewtheorem{cor}{Corollary}[section]{\bfseries}{\normalfont}
\spnewtheorem{rem}{Remark}[section]{\bfseries}{\normalfont}

\numberwithin{equation}{section}

\begin{document}

\title{A Homogeneous Second-Order Descent Ascent Algorithm for Nonconvex-Strongly Concave Minimax Problems\thanks{This work is supported by National Key  R \& D Program of China (Nos. 2025YFA1017801 and 2025YFA1017800), the National Natural Science Foundation of China under the grant 12471294.}
}

\titlerunning{A HSDA for NC-SC Minimax Problems}        

\author{Jia-Hao Chen \and Zi Xu \and Hui-Ling Zhang}


\institute{Jia-Hao Chen\at
	Department of Mathematics,  College of Sciences, Shanghai University, Shanghai 200444, P.R.China.\\ \email{chenjiahao@shu.edu.cn}
	\and	Zi Xu\at
	Department of Mathematics, College of Sciences, Shanghai University, Shanghai 200444, P.R.China. \\
	Corresponding author. \email{xuzi@shu.edu.cn}           
		\and Hui-Ling Zhang\at
	LSEC, ICMSEC, Academy of Mathematics and Systems Science, Chinese Academy of Sciences, Beijing 100190, China. \\ \email{zhanghl1209@shu.edu.cn}
}

\date{Received: date / Accepted: date}

\maketitle

\begin{abstract}
This paper introduces a novel Homogeneous Second-order Descent Ascent (HSDA) algorithm for nonconvex-strongly concave minimax optimization problems. At each iteration, HSDA uniquely computes a search direction by solving a homogenized eigenvalue subproblem built from the gradient and Hessian of the objective function. This formulation guarantees a descent direction with sufficient negative curvature even in near-positive-semidefinite Hessian regimes—a key feature that enhances escape from saddle points. We prove that HSDA finds an $\mathcal{O}(\varepsilon,\sqrt{\varepsilon})$-second-order stationary point within $\tilde{\mathcal{O}}(\varepsilon^{-3/2})$ iterations, matching the optimal $\varepsilon$-order iteration complexity among second-order methods for this problem class. To address large-scale applications, we further design an inexact variant (IHSDA) that preserves the single-loop structure while solving the subproblem approximately via a Lanczos procedure. With high probability, IHSDA achieves the same $\tilde{\mathcal{O}}(\varepsilon^{-3/2})$ iteration complexity and attains an $\mathcal{O}(\varepsilon,\sqrt{\varepsilon})$-second-order stationary point, with the total Hessian‑vector product cost bounded by $\tilde{\mathcal{O}}(\varepsilon^{-7/4})$. Experiments on synthetic minimax problems and adversarial training tasks confirm the practical effectiveness and robustness of the proposed algorithms.
\end{abstract}

\section{Introduction}

In this paper, we consider the following unconstrained minimax problem:
\begin{equation}\label{eq:minimax}
    \min_{x \in \mathbb{R}^{n}} \max_{y \in \mathbb{R}^{m}} f(x,y),\tag{P}
\end{equation}
where  $f(x,y):\mathbb{R}^{n} \times \mathbb{R}^{m} \to \mathbb{R} $ is a continuously differentiable function, which is strongly concave in $y$, but possibly nonconvex in $x$. For convenience, we denote
\begin{equation}\label{eq:value}
    \mathcal{F}(x) := \max_{y \in \mathbb{R}^{m}} f(x,y).
\end{equation}
Such a structure captures a wide range of machine learning applications, including adversarial training and distributionally robust optimization \cite{gao2023distributionally,sanjabi2018convergence,sinha2017certifying},  reinforcement learning, domain adaptation, and AUC maximization \cite{ganin2016domain,qiu2020single,ying2016stochastic}.

To solve the minimax problem \eqref{eq:minimax}, three main classes of optimization algorithms have been developed: zeroth-order, first-order, and second-order methods, which utilize the function value, gradient, and Hessian of the objective function, respectively. Compared to zeroth- and first-order approaches, second-order methods have garnered considerable attention owing to their faster convergence rates. Moreover, they are more effective at escaping saddle points and avoiding poor local minima, thereby increasing the likelihood of converging to a globally optimal solution. This paper focuses on second-order optimization algorithms for solving \eqref{eq:minimax}.

For nonconvex-strongly concave minimax problems, first-order algorithms can obtain an $\varepsilon$-first-order stationary point in 
$\tilde{\mathcal{O}}(\kappa_y^2 \varepsilon^{-2})$ iterations \cite{jin2020local,lin2020gradient,lu2020hybrid,rafique2022weakly,xu2023unified}, 
where $\kappa_y$ denotes the condition number of $f(x,\cdot)$. 
Acceleration frameworks further improve the iteration complexity to
$\tilde{\mathcal{O}}(\sqrt{\kappa_y}\,\varepsilon^{-2})$ \cite{lin2020near,zhang2021complexity,Li2021ComplexityLB}. 

There are few studies on second-order algorithms for solving nonconvex-strongly concave minimax optimization problems \eqref{eq:minimax}. Existing second-order algorithms can be divided into two categories, i.e., cubic regularization Newton type algorithms \cite{luo2022finding,chen2021cubic} and trust-region type algorithms \cite{yao2024two,wang2025gradient}.  Building upon the cubic regularization (CR) framework, Luo et al. \cite{luo2022finding} proposed the Minimax Cubic Newton (MCN) algorithm, which alternates between a cubic-regularized Newton step in the minimization variable and an ascent step in the maximization variable, achieving an iteration complexity of $\mathcal{O}(\varepsilon^{-3/2})$ to reach an $\mathcal{O}(\varepsilon,\sqrt{\varepsilon})$-second-order stationary point. They further introduced an inexact variant (IMCN) that solves the cubic subproblem via gradient-based iterations and approximates Hessian inverse operations using Chebyshev polynomial expansions, relying solely on Hessian-vector products. Within the same line of work, Chen et al. \cite{chen2021cubic} developed the Inexact Cubic-LocalMinimax (ICLM) algorithm, which attains the same order of iteration complexity. In the trust-region family of methods, Yao and Xu \cite{yao2024two} proposed MINIMAX-TR, a fixed-radius inexact trust-region method that finds an $\mathcal{O}(\varepsilon,\sqrt{\varepsilon})$-second-order stationary point within $\mathcal{O}(\varepsilon^{-3/2})$ iterations. To enhance practical performance, they also designed MINIMAX-TRACE, which adaptively adjusts the trust-region radius through contraction and expansion steps while maintaining the same theoretical iteration complexity. More recently, Wang and Xu \cite{wang2025gradient} introduced a gradient norm regularized trust-region method (GRTR) and a Levenberg-Marquardt type negative-curvature method (LMNegCur) for nonconvex-strongly concave minimax problems. GRTR achieves an iteration complexity of $\tilde{\mathcal{O}}(\varepsilon^{-3/2})$, and its inexact variant IGRTR preserves this rate while reducing Hessian-vector product computations to be $\tilde{\mathcal{O}}(\varepsilon^{-7/4})$. LMNegCur and its inexact counterpart ILMNegCur offer analogous convergence guarantees.

Collectively, these advances highlight the ongoing development of second-order methods tailored to nonconvex-strongly concave minimax optimization and motivate the algorithm design pursued in this work.

\subsection{Contributions}
In this paper, we propose a homogeneous second-order descent ascent (HSDA) algorithm whose outer iteration solves a single homogenized eigenvalue subproblem—constructed from the gradient and Hessian of the value function—to obtain an iteration direction. This homogenized formulation guarantees, even when the Hessian is nearly positive semidefinite, a descent direction with sufficient negative curvature for the value function. We prove that HSDA finds an $\mathcal{O}(\varepsilon,\sqrt{\varepsilon})$-second-order stationary point within  $\tilde{\mathcal{O}}(\varepsilon^{-3/2})$ iterations, matching the best known iteration complexity for second-order methods in this setting \cite{chen2021cubic,luo2022finding,wang2025gradient}. For large-scale problems, we develop an inexact version (IHSDA) that approximately solves the homogenized eigenvalue subproblem via a Lanczos procedure with a carefully controlled residual. IHSDA preserves the same $\tilde{\mathcal{O}}(\varepsilon^{-3/2})$ outer iteration complexity and, with high probability, reaches an 
$\mathcal{O}(\varepsilon,\sqrt{\varepsilon})$-second-order stationary point, while its total Hessian-vector product computations are upper bounded by $\tilde{\mathcal{O}}(\varepsilon^{-7/4})$.

Unlike recent inexact trust-region schemes (e.g., IGRTR and ILMNegCur \cite{wang2025gradient}) that require solving both a regularized Newton system and an $n$-dimensional extremal-eigenvalue problem, while HSDA only solves a single $(n+1)$-dimensional extremal-eigenvalue problem in a lifted space at each iteration. We further show (in Section \ref{sec_ihsda}) that the homogenized eigenvalue subproblems typically exhibit better conditioning than the regularized Newton/trust-region systems underlying IGRTR/ILMNegCur. Hence, HSDA and IHSDA offer an alternative second-order framework whose inner subproblem is structurally simpler, yet achieves the same $\tilde{\mathcal{O}}(\varepsilon^{-7/4})$ Hessian-vector product computations as  IGRTR and ILMNegCur.

{\bfseries Notation}.
We adopt the following notation throughout the paper: $[a; b]$ and $[a, b]$ denote vertical and horizontal concatenation, respectively; $\operatorname{sgn}(\cdot)$ is the sign function, defined by $\operatorname{sgn}(a) = -1$ if $a < 0$ and $\operatorname{sgn}(a) = 1$ if $a \geqslant 0$. For a vector $a \in \mathbb{R}^n$ and $0 \leqslant j \leqslant n$, $a_{[1:j]}$ denotes the subvector formed by its first $j$ entries. The symbol $\|\cdot\|$ denotes the Euclidean norm for vectors and the induced $\ell_2$ operator norm for matrices. The eigenvalues of a matrix $A \in \mathbb{R}^{n \times n}$ are ordered as $\lambda_1(A), \lambda_2(A), \dots, \lambda_{\max}(A)$ in nondecreasing order. The identity matrix of dimension $n$ is written as $I_n$, or simply $I$ when the dimension is clear. For a function $f(x, y): \mathbb{R}^n \times \mathbb{R}^m \to \mathbb{R}$, $\nabla_x f(x, y)$ and $\nabla_y f(x, y)$ denote its partial gradients with respect to $x$ and $y$, respectively; the full gradient is $\nabla f(x, y) := (\nabla_x f(x, y), \nabla_y f(x, y))$. Second-order partial derivatives are denoted by $\nabla_{xx}^2 f(x, y)$, $\nabla_{xy}^2 f(x, y)$, $\nabla_{yx}^2 f(x, y)$, and $\nabla_{yy}^2 f(x, y)$. Complexity notations $\mathcal{O}(\cdot), \Omega(\cdot), \Theta(\cdot)$ hide only absolute constants independent of problem parameters, while $\tilde{\mathcal{O}}(\cdot)$ additionally hides logarithmic factors. We also define the value function $\mathcal{F}(x) := \max_{y \in \mathbb{R}^m} f(x, y)$, the maximizer $y^\ast(x) := \arg\max_{y \in \mathbb{R}^m} f(x, y)$, the partial gradient $g(x, y) := \nabla_x f(x, y)$, and  $H(x, y) := \bigl[ \nabla_{xx}^2 f - \nabla_{xy}^2 f (\nabla_{yy}^2 f)^{-1} \nabla_{yx}^2 f \bigr](x, y)$.

\section{A Homogeneous Second-Order Descent Ascent Algorithm}\label{HSDA_al}
In this section, we propose a Homogeneous Second-order Descent Ascent (HSDA) algorithm for solving the nonconvex-strongly concave minimax problem \eqref{eq:minimax}. HSDA is inspired by the Homogeneous Second-order Descent Method (HSODM)~\cite{zhang2025homogeneous}, a second-order framework originally designed for unconstrained minimization problems of the form $\min_{x\in\mathbb{R}^n} \mathcal{F}(x)$. At each iteration, HSODM obtains a search direction by solving a homogenized eigenvalue subproblem of the form:
\begin{equation*}
  \begin{array}{ll}
    \mathop{\min}\limits_{\|[u;v]\|\leqslant 1} &
    \begin{bmatrix} u \\ v \end{bmatrix}^{\!\top}
    \begin{bmatrix}
      \nabla^2\mathcal{F}(x) & \nabla \mathcal{F}(x)\\[2pt]
      \nabla \mathcal{F}(x)^{\top} & -\alpha
    \end{bmatrix}
    \begin{bmatrix} u \\ v \end{bmatrix}, \\
  \end{array}
\end{equation*}
where $\alpha \geqslant 0$ is a prescribed parameter. 
As demonstrated in~\cite{zhang2025homogeneous}, the resulting subproblem is an eigenvalue problem, where the optimal solution $[u_t; v_t]$ is a unit eigenvector associated with the smallest eigenvalue of the homogenized matrix. Building on this framework, we propose a HSDA algorithm as a generalized and inexact extension of HSODM tailored to the minimax setting in \eqref{eq:minimax} with $\mathcal{F}(x) = \max_{y\in\mathbb{R}^{m}} f(x,y)$. At each iteration, HSDA incorporates two key algorithmic components:
\begin{itemize}
  \item Inexact inner maximization via gradient ascent: Given $x_t$, we perform $N_t$ steps of
  Nesterov's accelerated gradient ascent~\cite{nesterov2018lectures} on $y$ to obtain an approximate maximizer
  \begin{equation*}\label{eq:inner-y}
    y_t \approx y^\ast(x_t) = \arg\max_{y\in\mathbb{R}^m} f(x_t,y).
  \end{equation*}
  Using $y_t$, we then construct the corresponding inexact first- and second-order information of $\mathcal{F}$ at $x_t$:
  \begin{equation*}\label{eq:gH}
    g_t := g(x_t,y_t) , \quad H_t := H(x_t,y_t) .
  \end{equation*}
  \item Direction generation via homogenized eigenvector computation: Given the inexact gradient estimate $g_t$, Hessian estimate $H_t$, and a parameter $\alpha>0$, we compute the search direction $[u_t; v_t]$ with $u_t\in \mathbb{R}^n$ and $v_t\in \mathbb{R}$ by solving the following homogenized eigenvalue subproblem: 
  \begin{equation}\label{eq:subprob}
  \min_{\|[u;v]\|\leqslant 1}
  \begin{bmatrix}u\\ v\end{bmatrix}^{\!\top}
  G_t(\alpha)
  \begin{bmatrix}u\\ v\end{bmatrix},
  \qquad
  G_t(\alpha):=
  \begin{bmatrix}
    H_t & g_t\\
    g_t^{\top} & -\alpha
  \end{bmatrix}.
\end{equation}
The search direction $s_t$ is then defined as:
\begin{equation}\label{eq:dk-classification-clean}
  s_t =
  \begin{cases}
    \dfrac{u_t}{v_t}, & \qquad |v_t|\geqslant \omega ,
      \\[9pt]
    \operatorname{sgn}\!\big(-g_t^{\top}u_t\big)u_t, & \qquad |v_t|<\omega .
  \end{cases}
\end{equation}
\end{itemize}
The detailed algorithm is formally stated in Algorithm~\ref{alg:HSDA}.
\begin{algorithm}[!htbp]
  \caption{A Homogeneous Second-Order Descent Ascent (HSDA) Algorithm}
  \label{alg:HSDA}
  \begin{algorithmic}
    \STATE{\textbf{Step 1}: Input  $x_1$, $y_0$, $\eta_{1}>0$, $\eta_{2}>0$, $\omega\in(0,1/2)$, $\{N_t\geqslant 1\}$, $\varepsilon>0$, $\alpha>0$,  $\Lambda>0$ and set $t=1$.}
    \STATE{\textbf{Step 2}: Update $y_t$:}
    \STATE{\quad\textbf{(2a)}: Set $i=0$, $y_i^{t}= \tilde y_i^{t}=y_{t-1}$.}
    \STATE{\quad\textbf{(2b)}: Update $y_i^{t}$ and $\tilde y_i^{t}$:
    \begin{align*}
      y_{i+1}^{t} &= \tilde y_i^{t} + \eta_{1} \nabla_y f\big(x_t,\tilde y_i^{t}\big),\\
      \tilde y_{i+1}^{t} &= y_{i+1}^{t} + \eta_{2}\big(y_{i+1}^{t}-y_i^{t}\big).
    \end{align*}}
    \STATE{\quad\textbf{(2c)}: If $i \geqslant N_t-1$, set $y_t = y_{N_t}^{t}$; otherwise set $i = i+1$ and go to Step (2b).}
    \STATE{\textbf{Step 3}: Compute
    \[
      g_t = \nabla_x f(x_t,y_t),\qquad
      H_t = \big[\nabla^2_{xx}f - \nabla^2_{xy}f(\nabla^2_{yy}f)^{-1}\nabla^2_{yx}f\big](x_t,y_t)
    \]
    and solve the following homogeneous subproblem to obtain $[u_t;v_t]$:
    	 \begin{equation*}
    	[u_t;v_t]=\arg\min_{\|[u;v]\|\leqslant 1}
    		\begin{bmatrix}u\\ v\end{bmatrix}^{\!\top}
   \begin{bmatrix}
   	H_t & g_t\\[2pt]
   	g_t^{\top} & -\alpha
   \end{bmatrix}
    		\begin{bmatrix}u\\ v\end{bmatrix}.
    	\end{equation*}}
    \STATE{\textbf{Step 4}: Update the direction $s_t$:
    \begin{equation*}
    	s_t =
    	\begin{cases}
    		\dfrac{u_t}{v_t}, & \qquad |v_t|\geqslant \omega ,
    		\\[9pt]
    		\operatorname{sgn}\!\big(-g_t^{\top}u_t\big)u_t, & \qquad |v_t|<\omega .
    	\end{cases}
    \end{equation*}
    }
  \STATE{\textbf{Step 5}:  If some stationary condition or stopping criterion is satisfied, set $x_{t+1} = x_t + s_t$ and \textbf{terminate}. Otherwise, compute $\tau_t = \Lambda/\|s_t\|$, update $x_{t+1} = x_t + \tau_t s_t$, set $t = t+1$ and go to Step 2.}
  \end{algorithmic}
\end{algorithm}
In the following subsection, we establish that the proposed HSDA algorithm attains an $\mathcal{O}(\varepsilon,\sqrt{\varepsilon})$-second-order stationary point of $\mathcal{F}(x)$ for problem \eqref{eq:minimax} with an iteration complexity of $\tilde{\mathcal{O}}(\varepsilon^{-3/2})$.
\subsection{Complexity Analysis}
Throughout this paper, we work under the following assumptions on $f(x,y)$, which ensure that the value function $\mathcal{F}(x)$ is well-defined and has the required  smoothness.
\begin{ass}\label{as.1}
\(f(x,y)\) satisfies the following assumptions:
\begin{itemize}
	\item[(1)] The function \(f(x,y)\) is assumed to be \(\mu\)-strongly concave in $y$ for any fixed $x$.
	\item[(2)] We assume the gradient \(\nabla f(\cdot,\cdot)\) is \(\ell_1\)-Lipschitz continuous in the joint variable: for all \((x,y),(x',y')\in\mathbb{R}^n\times\mathbb{R}^m\), 
	\begin{equation*}
		\big\|\nabla f(x,y)-\nabla f(x',y')\big\| \leqslant \ell_1 \big\|(x,y)-(x',y')\big\|.
	\end{equation*}
	\item[(3)]  We assume that the Hessian blocks \(\nabla^2_{xx}f(x,y)\), \(\nabla^2_{xy}f(x,y)\), \(\nabla^2_{yx}f(x,y)\), and \(\nabla^2_{yy}f(x,y)\) are \(\ell_2\)-Lipschitz continuous.
	\item[(4)] The value function \(\mathcal{F}(x)=\max_{y\in\mathbb{R}^m} f(x,y)\) has a finite lower bound \(\mathcal{F}_{\inf}\).
\end{itemize}
\end{ass}
Under these assumptions, Lemma~\ref{lem2.1} establishes the Lipschitz continuity of 
\(\nabla\mathcal{F}\) and \(\nabla^{2}\mathcal{F}\).
\begin{lem}[\cite{chen2021cubic}]\label{lem2.1}
 Under Assumption~\ref{as.1}, the following properties hold:
 \begin{itemize}
 	\item[(1)] The maximizer \(y^*(x)=\arg\max_{y\in\mathbb{R}^m} f(x,y)\)
 	is well-defined for every $x$ and is \(\kappa\)-Lipschitz continuous.
 	\item[(2)] \(\nabla\mathcal{F}(x)=\nabla_x f\big(x,y^*(x)\big)\),
 	and is \(L_1\)-Lipschitz continuous with \(L_1:=(\kappa+1)\ell_1\) and \(\kappa:=\ell_1/\mu\).
 	\item[(3)] \(H(x,y)\) is \(L_H\)-Lipschitz continuous with \(L_H:=\ell_2(1+\kappa)^2\) and \(\|H(x,y)\|\leqslant L_1\).
 	\item[(4)] It holds that 	\(\nabla^2\mathcal{F}(x)=H\big(x,y^*(x)\big)\). Consequently, the Hessian \(\nabla^2\mathcal{F}(x)\) is Lipschitz continuous with constant \(L_2:=\ell_2(1+\kappa)^3\).
\end{itemize}	
\end{lem}
Lemma \ref{lem2.1} directly yields the upper bounds presented in Lemma~\ref{lem2.2}. 
\begin{lem}[\cite{nesterov2018lectures}]\label{lem2.2}
Under Assumption~\ref{as.1}, for all $x,x'\in\mathbb{R}^n$ the gradient and Hessian of 
$\mathcal{F}$ satisfy the following inequalities:
\begin{equation*}
\begin{aligned}
&\left\| \nabla \mathcal{F}(x') - \nabla \mathcal{F}(x) - \nabla^{2}\mathcal{F}(x)(x'-x) \right\|
\leqslant \frac{L_{2}}{2}\|x'-x\|^{2},\\
&\left| \mathcal{F}(x') - \mathcal{F}(x) - \nabla \mathcal{F}(x)^{\top}(x'-x)
- \frac{1}{2}(x'-x)^{\top}\nabla^{2}\mathcal{F}(x)(x'-x) \right|
\leqslant \frac{L_{2}}{6}\|x'-x\|^{3},
\end{aligned}
\end{equation*}
where $L_{2}$ is the Lipschitz constant of $\nabla^{2}\mathcal{F}$ given in Lemma~\ref{lem2.1}.
\end{lem}
We next recall the standard definitions of first- and second-order stationary points used in minimax optimization~\cite{luo2022finding}.

\begin{definition}\label{def.1}
	A point \(x\) is an \(\varepsilon\)-first-order stationary point of  \(\mathcal{F}(x)\)
 when \(\|\nabla \mathcal{F}(x)\|\leqslant \varepsilon\).
\end{definition}

\begin{definition}\label{def.2}
A point \(x\) is an \(\mathcal{O}(\varepsilon,\sqrt{\varepsilon})\)-second-order stationary point of \(\mathcal{F}(x)\) when
\begin{equation*}
\|\nabla \mathcal{F}(x)\|\leqslant c_{1}\varepsilon \quad\text{and}\quad \nabla^2\mathcal{F}(x)\succcurlyeq -c_{2}\sqrt{\varepsilon}I,
\end{equation*}
where positive constants $c_{1}$ and $c_{2}$ do not depend on \(\varepsilon\).
\end{definition}
The following lemma demonstrates that by selecting an appropriate step size and performing sufficiently many inner gradient-ascent updates on $y$, we can obtain approximations of the gradient and Hessian of the value function within a desired accuracy.
\begin{lem}[\cite{wang2025gradient}]\label{lem.2} 
Suppose that Assumption~\ref{as.1} holds. For any \(\varepsilon_1,\varepsilon_2>0\), set the inner ascent step sizes as \(\eta_1=1/\ell_1\), \(\eta_{2}= (\sqrt{\kappa}-1)/(\sqrt{\kappa}+1)\)
and define $A=\min\left\{\frac{\varepsilon_1}{\ell_1},\ \frac{\varepsilon_2}{2L_H}\right\}$, where \(L_H\) is the Lipschitz constant of \(H(x,y)\) given in Lemma~\ref{lem2.1}.
If the iteration counts \(\{N_t\}\) for \(y\)-updates in Algorithm~\ref{alg:HSDA} satisfy
\begin{equation*}
\begin{aligned}
N_1 & \geqslant 2\sqrt{\kappa}
\log\!\left(\frac{\sqrt{\kappa+1}\|y_0 - y^*(x_1)\|}{A}\right),\\
N_t & \geqslant 2\sqrt{\kappa}
\log\!\left(\frac{\sqrt{\kappa+1}\big(A+\kappa\|x_{t}-x_{t-1}\|\big)}{A}\right),
\qquad t\geqslant 2,
\end{aligned}
\end{equation*}
then for every \(t\geqslant 1\) the following error bounds hold:
\begin{equation*}
	\left\|y_t-y^*\left(x_t\right)\right\| \leqslant A,
	\qquad
	\big\|\nabla\mathcal{F}(x_t)-g_t\big\|\leqslant \varepsilon_1,
	\qquad
	\big\|\nabla^2\mathcal{F}(x_t)-H_t\big\|\leqslant \varepsilon_2.
\end{equation*}
\end{lem}
The following lemma characterizes the optimality conditions for the homogenized eigenvalue subproblem \eqref{eq:subprob}.
\begin{lem}[\cite{zhang2025homogeneous}]\label{lem.4}
The vector $[u_t; v_t]$ is a solution to the homogenized eigenvalue subproblem \eqref{eq:subprob} if and only if there exists a dual scalar $\delta_t$ such that	
\begin{subequations}\label{opt}
\begin{align}
\begin{bmatrix}
H_t + \delta_t I & g_t \\
g_t^{\top} & -\alpha + \delta_t
\end{bmatrix} &\succcurlyeq 0, \label{opt.1}\\
(H_t+\delta_t I)u_t = -v_tg_t,
\quad
g_t^{\top} u_t &= v_t(\alpha-\delta_t), \label{opt.2}\\
\delta_t \geqslant \alpha > 0, \quad \big\|[u_t; v_t]\big\| &= 1. \label{opt.3}
\end{align}
\end{subequations}
Furthermore, $-\delta_t$ equals the smallest eigenvalue of the homogenized matrix $G_t(\alpha)$, i.e., $-\delta_t = \lambda_{1}\!\big(G_t(\alpha)\big),$ and $[u_t; v_t]$ is a corresponding unit eigenvector. Moreover, when $g_t \neq 0$, the inequality $\delta_t \geqslant \alpha > 0$ in \eqref{opt.3} can be strengthened to the strict form $\delta_t > \alpha > 0$.
\end{lem}

We now proceed to analyze the iteration complexity of the proposed HSDA algorithm. Our analysis begins by establishing a descent property for the case where $|v_t| \leqslant \sqrt{1/(1+\Lambda^{2})}$.
\begin{lem}\label{lem:step-decrease}
Suppose that Assumption~\ref{as.1} holds. Assume $\Lambda \leqslant \sqrt{2}/2$ 
and $\omega\in(0,1/2)$. For the case $|v_t| \leqslant \sqrt{1/(1+\Lambda^{2})}$, we have 
\begin{equation}
\mathcal{F}(x_{t+1})-\mathcal{F}(x_t)\leqslant \Lambda \varepsilon_1+\frac{\Lambda^{2}}{2}(\varepsilon_2-\alpha) +\frac{L_{2}}{6} \Lambda^{3}.
\end{equation}
\end{lem}

\begin{proof}
We first prove that when $|v_t| \leqslant \sqrt{1/(1+\Lambda^{2})}$, the direction $s_t$ satisfies:
\begin{equation}\label{s_t_value}
	\|s_t\| \geqslant \Lambda.
\end{equation}
Since when $|v_t| < \omega$, according to $s_t = \operatorname{sgn}(-g_t^{\top}u_t)u_t$ in Algorithm~\ref{alg:HSDA} and~\eqref{opt.3}, we have
\[
\|s_t\| = \|u_t\|= \sqrt{1 - |v_t|^2}\geqslant \sqrt{1 - \omega^2}\geqslant {\sqrt{3}}/{2}\geqslant \Lambda .
\]
When $\omega \leqslant |v_t| \leqslant \sqrt{1/(1+\Lambda^2)}$, according to Algorithm~\ref{alg:HSDA} and ~\eqref{opt.3}, we get
\[\|s_t\|= \|{u_t}/{v_t}\|= {\sqrt{1 - |v_t|^2}}/{|v_t|}\geqslant \Lambda . \]
Therefore, \eqref{s_t_value} holds, and thus $\tau_t = \Lambda / \|s_t\| \in (0,1].$	

Denote $E_t := \tau_t g_{t}^{\top} s_t + \frac{\tau_t^2}{2} s_t^{\top} H_{t} s_t$. 
It follows from Step 5 in Algorithm~\ref{alg:HSDA}, the $L_{2}$-Lipschitz continuity
of $\nabla^2 \mathcal{F}(x)$, and Lemma~\ref{lem.2} that
	\begin{align}\label{guding.3}
        &~\mathcal{F}(x_{t+1}) - \mathcal{F}(x_t) \nonumber\\
        \leqslant&~ \tau_t \nabla \mathcal{F}(x_{t})^{\top} s_t 
        + \frac{\tau_t^2}{2} s_t^{\top} \nabla^2 \mathcal{F}(x_{t}) s_t 
        + \frac{L_{2}}{6} \tau_t^3 \|s_t\|^3 \nonumber\\
        =&~ E_t 
        + \tau_t (\nabla \mathcal{F}(x_{t}) - g_{t})^{\top} s_t
        + \frac{\tau_t^2}{2} s_t^{\top} \big(\nabla^2 \mathcal{F}(x_{t}) - H_{t}\big) s_t 
        + \frac{L_{2}}{6} \tau_t^3 \|s_t\|^3 \nonumber\\
        \leqslant&~ E_t 
        + \tau_t \|\nabla \mathcal{F}(x_{t}) - g_{t}\| \|s_t\|
        + \frac{\tau_t^2}{2} \|\nabla^2 \mathcal{F}(x_{t}) - H_{t}\| \|s_t\|^{2} 
        + \frac{L_{2}}{6} \tau_t^3 \|s_t\|^3 \nonumber\\
        \leqslant&~ E_t 
        + \varepsilon_1 \Lambda  
        + \frac{\varepsilon_2}{2} \Lambda^{2} 
        + \frac{L_{2}}{6} \Lambda^{3}. 
    \end{align}
When $|v_t|<\omega$ and $g_t \ne 0$, by \eqref{opt.2},~\eqref{opt.3} and $s_t = \operatorname{sgn}(-g_t^{\top}u_t)u_t$ in Algorithm~\ref{alg:HSDA},
we obtain
\begin{equation}
    \label{guding.1}
    s_t^{\top} H_t s_t = -\delta_t \|s_t\|^2 - v_t^2 (\alpha - \delta_t), 
    \quad
    g_t^{\top} s_t = |v_t| (\alpha - \delta_t).
\end{equation}
Therefore, using \eqref{guding.1}, we get
\begin{equation*}
    \begin{aligned}
        E_t 
        &= \tau_t g_{t}^{\top} s_t + \frac{\tau_t^2}{2} s_t^{\top} H_{t} s_t\\
        &= \tau_t |v_t|(\alpha - \delta_t)
           - \frac{\tau_t^2}{2} \delta_t \|s_t\|^2
           - \frac{\tau_t^2}{2} v_t^2 (\alpha - \delta_t) \\
        &\leqslant \tau_t v_t^2 (\alpha - \delta_t)
           - \frac{\tau_t^2}{2} \delta_t \|s_t\|^2
           - \frac{\tau_t^2}{2} v_t^2 (\alpha - \delta_t) \\
        &= \Bigl(\tau_t - \frac{\tau_t^2}{2}\Bigr) v_t^2 (\alpha - \delta_t)
           - \frac{\tau_t^2}{2} \delta_t \|s_t\|^2,
    \end{aligned}
\end{equation*}
where the inequalities are derived from $|v_t|<\omega<1$ and $\alpha-\delta_t<0$. Since $\tau_t \in (0,1]$, we have $\tau_t - \tau_t^2/2 \geqslant 0$. Further combining this with \eqref{opt.3}, we get
\begin{equation}
    \label{guding.2}
\left(\tau_t - \frac{\tau_t^2}{2} \right) (\alpha - \delta_t) \leqslant 0.
\end{equation}
Furthermore, using $\tau_t=\Lambda/\|s_t\|$ and \eqref{guding.2}, we get
\begin{equation}
    \label{guding.4}
    E_t 
    \leqslant - \frac{\tau_t^2}{2} \delta_t \|s_t\|^2
    = -\delta_t \frac{\Lambda^2}{2}
    \leqslant - \frac{\Lambda^2}{2} \alpha.
\end{equation}
When $|v_t|<\omega$ and $g_t = 0$, then $E_t = \frac{\tau_t^2}{2} s_t^{\top} H_t s_t$.
In this case, it can be similarly proven that:
\begin{equation}
  \label{guding.4.1}
  E_t \leqslant - \frac{\Lambda^2}{2}\,\alpha .
\end{equation}

When $\omega \leqslant |v_t| \leqslant \sqrt{1/(1+\Lambda^2)}$,  we have $s_t = u_t / v_t$. Substituting this into \eqref{opt.2} yields
\begin{equation}
    \label{guding.5}
    s_t^{\top} H_t s_t = - g_t^{\top} s_t - \delta_t \|s_t\|^2, 
    \qquad 
    g_t^{\top} s_t = \alpha - \delta_t \leqslant 0.
\end{equation}
Since $\tau_t \in (0,1]$, it holds that $\tau_t - \tau_t^2/2 \geqslant 0$, and consequently
\begin{equation}
    \label{guding.6}
    \left( \tau_t - \frac{\tau_t^2}{2} \right) g_t^{\top} s_t \leqslant 0.
\end{equation}
Using \eqref{guding.5} and \eqref{guding.6}, we obtain
\begin{equation}
    \label{guding.8}
    \begin{aligned}
        E_t
        &= \tau_t g_{t}^{\top} s_t + \frac{\tau_t^2}{2} s_t^{\top} H_{t} s_t
        = \Bigl(\tau_t - \frac{\tau_t^2}{2}\Bigr) g_t^{\top} s_t
           - \frac{\tau_t^2}{2} \delta_t \|s_t\|^2\\
        &\leqslant - \frac{\tau_t^2}{2} \delta_t \|s_t\|^2
        \leqslant -\frac{\Lambda^2}{2} \alpha.
    \end{aligned}
\end{equation}
Combining \eqref{guding.3} with \eqref{guding.4}, \eqref{guding.4.1} and \eqref{guding.8} completes the proof.
\end{proof}

We now consider the case where $|v_t| > \sqrt{1/(1+\Lambda^{2})}$.
The subsequent lemma establishes explicit bounds for both $\|\nabla \mathcal{F}(x_{t+1})\|$ and $\nabla^2\mathcal{F}(x_{t+1})$.

\begin{lem}\label{lem2.6}
Under Assumption~\ref{as.1} and for the case $|v_t|>\sqrt{1/(1+\Lambda^{2})}$, 
we further assume $\Lambda \leqslant \sqrt{2}/2$. 
Then the following holds:
\begin{subequations}\label{eq:stationary-bounds}
  \begin{align}
    \big\|\nabla \mathcal{F}(x_{t+1})\big\|
    &\leqslant
      2(L_1 + \alpha)\Lambda^3
      + \frac{L_2}{2}\Lambda^2
      + (\varepsilon_2 + \alpha)\Lambda
      + 3 \varepsilon_1 ,
      \label{eq:stationary-bounds-grad}\\[2pt]
    \nabla^2 \mathcal{F}(x_{t+1})
    &\succcurlyeq
      -\Big\{2(L_1 + \alpha)\Lambda^2
      + \alpha + \varepsilon_{2}
      + L_H\big[(1+\kappa)\Lambda + 2A\big]\Big\}I .
      \label{eq:stationary-bounds-hess}
  \end{align}
\end{subequations}
\end{lem}
\begin{proof}
We first examine the case $g_t \neq 0$ to derive an upper bound for $\big\|\nabla \mathcal{F}(x_{t+1})\big\|$. The analysis begins by estimating $\|g_t\|$. Given the condition $|v_t|>\sqrt{1/(1+\Lambda^{2})}$, \eqref{opt.3} implies	
\begin{equation}\label{in_l}
  \|s_t\|
  = \Big\|\frac{u_t}{v_t}\Big\|
  = \frac{\sqrt{1-|v_t|^2}}{|v_t|}
  < \Lambda .
\end{equation}
Combining \eqref{opt.2} with the upper bound \eqref{in_l} yields
\begin{subequations}
  \begin{align}
    H_t s_t + g_t &= -\delta_t s_t, 
    \label{xiaozhi.1}\\[2pt]
    \delta_t - \alpha &= -g_t^{\top} s_t
    \leqslant \|g_t\| \|s_t\|
    \leqslant \Lambda \|g_t\| .
    \label{xiaozhi.2}
  \end{align}
\end{subequations}
Define the quadratic function
\[
h(m) := m^2 + \left( \frac{g_t^{\top} H_t g_t}{\|g_t\|^2} + \alpha \right) m - \|g_t\|^2 .
\]
The equation $h(m) = 0$ has two real roots of opposite signs; denote its positive root by $m_2$.
We now show that $h\left(\delta_t-\alpha\right) \geqslant 0$. To this end, consider the matrix
\[
Q(k) := 
\begin{bmatrix}
H_t + (k+\alpha) I & g_t \\
g_t^{\top}         & k
\end{bmatrix}.
\]
From the optimality condition, we have $Q(\delta_t - \alpha) \succcurlyeq 0$ and $\delta_t-\alpha>0$.
Applying the Schur complement with respect to the scalar block $k=\delta_t-\alpha$ gives
\[
H_t + \delta_t I - \frac{1}{\delta_t-\alpha} g_t g_t^{\top} \succcurlyeq 0.
\]
Premultiplying and postmultiplying the above inequality by the unit vector $g_t / \|g_t\|$ yields
\[
\frac{g_t^{\top} H_t g_t}{\|g_t\|^2} + \delta_t - \frac{\|g_t\|^2}{\delta_t-\alpha} \geqslant 0.
\]
Because $\delta_t - \alpha > 0$, multiplying both sides by $\delta_t - \alpha$ leads to
\begin{equation*}
(\delta_t-\alpha)\!\left(\frac{g_t^{\top} H_t g_t}{\|g_t\|^2} 
      + (\delta_t - \alpha) + \alpha\right) - \|g_t\|^2 
= (\delta_t-\alpha)^2 
   + \left(\frac{g_t^{\top} H_t g_t}{\|g_t\|^2} + \alpha\right)(\delta_t-\alpha)
   - \|g_t\|^2 \geqslant 0,
\end{equation*}
Hence $\delta_t - \alpha \geqslant m_2$ (since $m_2$ is the positive root of $h(m)=0$). 
Combining this with the bound from \eqref{xiaozhi.2}, we obtain
\begin{equation*}
    h(\Lambda \|g_t\|) = \Lambda^2 \|g_t\|^2 + \left( \frac{g_t^{\top} H_t g_t}{\|g_t\|^2} + \alpha \right) \Lambda \|g_t\| - \|g_t\|^2 \geqslant 0 .
\end{equation*}
Rearranging the inequality and using the condition $H_t \preccurlyeq L_1 I$ (i.e. ${g_t^{\top} H_t g_t}/{\|g_t\|^2} \leqslant L_1$) together with $\Lambda \leqslant \sqrt{2}/2$, we obtain
\begin{equation}
    \label{gk.shangjie}
    \|g_t\| 
    \leqslant \frac{\left( \frac{g_t^{\top} H_t g_t}{\|g_t\|^2} + \alpha \right) \Lambda}{1 - \Lambda^2} 
    \leqslant \frac{(L_1 + \alpha)\Lambda}{1 - \Lambda^2} 
    \leqslant 2(L_1 + \alpha)\Lambda .   
\end{equation}
From \eqref{in_l} and \eqref{xiaozhi.2}, it follows that
\begin{equation*}
  \delta_t \|s_t\|
  = \bigl(\alpha + (\delta_t - \alpha)\bigr)\|s_t\|
  \leqslant (\alpha + \Lambda \|g_t\|) \|s_t\|
  \leqslant \alpha \Lambda + \Lambda^2 \|g_t\| .
\end{equation*}
Combining this with \eqref{xiaozhi.1}, we obtain
\begin{equation*}
  \|H_t s_t + g_t\|
  = \delta_t \|s_t\|
  \leqslant \alpha \Lambda + \Lambda^2 \|g_t\| .
\end{equation*}
To bound $\|g_{t+1}\|$, we note that
    \begin{equation*}
            \begin{aligned}
                \|g_{t+1}\| 
                &\leqslant \|g_{t+1} - H_t s_t - g_t\| + \|H_t s_t + g_t\| \notag \\
                &\leqslant \|g_{t+1} - H_t s_t - g_t\| + \alpha \Lambda + \|g_t\| \Lambda^2  \\
                &\leqslant \|g_{t+1} - H_t s_t - g_t\| + \alpha \Lambda + 2(L_1 + \alpha)\Lambda^3,
            \end{aligned}
    \end{equation*}
where the last line uses inequality \eqref{gk.shangjie}. Furthermore, using \eqref{in_l} together with Lemmas~\ref{lem2.2} and \ref{lem.2}, we have
\begin{align}\label{gk.gap}
  \left\|g_{t+1}-H_t s_t-g_t\right\| \leqslant &~ \left\|\nabla \mathcal{F}\left(x_{t+1}\right)-\nabla^2 \mathcal{F}\left(x_t\right) s_t-\nabla \mathcal{F}\left(x_t\right)\right\|  +\left\|g_{t+1}-\nabla \mathcal{F}\left(x_{t+1}\right)\right\| \nonumber\\
  &~ +\left\|\nabla^2 \mathcal{F}\left(x_t\right)-H_t\right\|\left\|s_t\right\| + \left\|\nabla \mathcal{F}\left(x_t\right)-g_t\right\| \nonumber\\
  \leqslant &~ \frac{L_2}{2}\left\|s_t\right\|^2+\varepsilon_2\left\|s_t\right\|+2 \varepsilon_1 \nonumber\\
  \leqslant &~ \frac{L_2}{2}\Lambda^2+\varepsilon_2\Lambda+2 \varepsilon_1 .
\end{align}
Consequently,
\begin{equation*}
    \|g_{t+1}\| \leqslant 2(L_1 + \alpha)\Lambda^3 + \frac{L_2}{2}\Lambda^2 + (\varepsilon_2 + \alpha)\Lambda+2 \varepsilon_1 .
\end{equation*}
Finally, applying Lemma~\ref{lem.2} gives
\begin{equation}\label{exact_upper_1}
    \left\|\nabla \mathcal{F}\left(x_{t+1}\right)\right\| \leqslant \left\|g_{t+1}\right\| + \varepsilon_{1} \leqslant 2(L_1 + \alpha)\Lambda^3 + \frac{L_2}{2}\Lambda^2 + (\varepsilon_2 + \alpha)\Lambda+ 3 \varepsilon_1 .
\end{equation}

We now establish a lower bound for the Hessian $\nabla^2 \mathcal{F}(x_{t+1})$. 
From \eqref{opt.1}, we have $H_t + \delta_t I \succcurlyeq 0$. Combining this with \eqref{xiaozhi.2} and \eqref{gk.shangjie} yields
    \begin{equation}
        \label{Hk.xiajie}
        H_t \succcurlyeq -\delta_t I \succcurlyeq -(\Lambda \|g_t\| + \alpha)I 
            \succcurlyeq -2(L_1 + \alpha)\Lambda^2 I - \alpha I .
    \end{equation}
To bound $H_{t+1}$, we first use the $L_H$-Lipschitz continuity of $H(x,y)$:
\begin{equation*}\label{eq:Htplus1-first}
    \begin{aligned}
        H_{t+1} 
        &\succcurlyeq H_t - \|H_{t+1} - H_t\| I \\
        &\succcurlyeq H_t - L_H\bigl(\|x_{t+1}-x_t\| + \|y_{t+1}-y_t\|\bigr) I .
    \end{aligned}
\end{equation*}
Since $x_{t+1} = x_t + s_t$ with $\|s_t\| < \Lambda$, and using  the $\kappa$-Lipschitz continuity of $y^{\star}(x)$ together with the bound
$\|y_t - y^{\star}(x_t)\| \leqslant A$  from Lemma~\ref{lem.2}, we obtain
\begin{align}
	\label{eq:Htplus1-second}
    H_{t+1}
    &\succcurlyeq H_t - L_H\Bigl(\|s_t\|
    + \|y_{t+1}-y^{\star}(x_{t+1})\|
    + \|y^{\star}(x_{t+1})-y^{\star}(x_t)\|
    + \|y^{\star}(x_t)-y_t\|\Bigr) I \nonumber\\
    &\succcurlyeq H_t - L_H\bigl[(1+\kappa)\|s_t\| + 2A\bigr] I \nonumber\\
    &\succcurlyeq H_t - L_H\bigl[(1+\kappa)\Lambda + 2A\bigr] I .
\end{align}
Substituting the lower bound for $H_t$ from \eqref{Hk.xiajie} into \eqref{eq:Htplus1-second} gives
    \begin{equation}
        H_{t+1} \succcurlyeq -2(L_1 + \alpha)\Lambda^2 I - \alpha I -  L_H\left[(1+\kappa)\Lambda+2A\right] I .
    \end{equation}
Finally, applying the Hessian approximation error from Lemma~\ref{lem.2}, we obtain the desired lower bound for the exact Hessian:
\begin{align}\label{exact_lower_1}
    \nabla^2 \mathcal{F}(x_{t+1})
    &\succcurlyeq H_{t+1} - \varepsilon_{2} I \nonumber\\
    &\succcurlyeq -\Bigl(2(L_1 + \alpha)\Lambda^2 + \alpha + \varepsilon_{2}
    + L_H\bigl[(1+\kappa)\Lambda+2A\bigr]\Bigr) I.
\end{align}

We now consider the case $g_t = 0$. The argument for establishing an upper bound on  $\big\|\nabla \mathcal{F}(x_{t+1})\big\|$ proceeds analogously to the case $g_t \ne 0$ with the simplification that $g_t = 0$. 
Therefore, we obtain
\begin{equation*}
    \|g_{t+1}\| \leqslant \frac{L_2}{2}\Lambda^2 + (\varepsilon_2 + \alpha)\Lambda + 2 \varepsilon_1 .
\end{equation*}
Applying Lemma~\ref{lem.2} then yields
\begin{equation}\label{exact_upper_2}
    \left\|\nabla \mathcal{F}\left(x_{t+1}\right)\right\|
    \leqslant \left\|g_{t+1}\right\| + \varepsilon_{1}
    \leqslant \frac{L_2}{2}\Lambda^2 + (\varepsilon_2 + \alpha)\Lambda + 3 \varepsilon_1 .
\end{equation}
Combining \eqref{exact_upper_1} with \eqref{exact_upper_2}, we have \eqref{eq:stationary-bounds-grad}.
The lower bound for the Hessian proceeds analogously.  Under $g_t = 0$, we have
\begin{equation*}
        H_{t+1} \succcurlyeq - \alpha I -  L_H\bigl[(1+\kappa)\Lambda+2A\bigr] I .
\end{equation*}
Consequently,
\begin{equation}\label{exact_lower_2}
    \nabla^2 \mathcal{F}(x_{t+1})
    \succcurlyeq -\Bigl( \alpha + \varepsilon_{2}
    + L_H\bigl[(1+\kappa)\Lambda+2A\bigr]\Bigr) I .
\end{equation}
Finally, combining \eqref{exact_lower_1}  with the Hessian bound \eqref{exact_lower_2} establishes \eqref{eq:stationary-bounds-hess}, which completes the proof.
\end{proof}

We are now prepared to establish the iteration complexity of the HSDA algorithm. Let $\varepsilon>0$ be the target accuracy. We define the first iteration at which an $\mathcal{O}(\varepsilon,\sqrt{\varepsilon})$-second-order stationary point is reached as
\begin{equation}\label{eq:T-eps}
T(\varepsilon)
:= \min\Bigl\{t \,\Bigm|\,
\|\nabla \mathcal{F}(x_{t+1})\|\leqslant c_1\varepsilon
\ \text{and}\
\nabla^2\mathcal{F}(x_{t+1}) \succcurlyeq -c_2\sqrt{\varepsilon}\,I
\Bigr\},
\end{equation}
where $c_1$ and $c_2$ are positive constants independent of $\varepsilon$. In parallel, we introduce a verifiable stopping index based on the eigenvector component $v_t$:
\begin{equation}\label{eq:Ttilde-eps}
\widetilde T(\varepsilon)
:= \min\Bigl\{t \,\Bigm|\,
|v_t|>\sqrt{1/(1+\Lambda^2)}
\Bigr\}.
\end{equation}
The following theorem shows that once $|v_t|>\sqrt{1/(1+\Lambda^2)}$, the next iterate $x_{t+1}=x_{t}+s_{t}$ is already an  $\mathcal{O}(\varepsilon,\sqrt{\varepsilon})$-second-order stationary point; consequently, 
$T(\varepsilon)\leqslant \widetilde T(\varepsilon)$.
\begin{thm}\label{the.1}
Suppose Assumption~\ref{as.1} holds, and set the parameters as
$
\alpha=\sqrt{L_2\varepsilon},\,
\varepsilon_1=\varepsilon/12,\,
\varepsilon_2=\sqrt{L_2\varepsilon}/12,\,
\Lambda=\sqrt{\varepsilon/L_2},\,
\omega\in(0,1/2)
$
with $0<\varepsilon\leqslant \min\{L_2/2,\,1\}$. Then the iterate $x_{\widetilde T(\varepsilon)+1}$ satisfies
\begin{equation}\label{eq:sosp-bounds}
  \begin{aligned}
    &\|\nabla \mathcal{F}(x_{\widetilde T(\varepsilon)+1})\|
    \leqslant\Bigl(\frac{\sqrt{2}\,L_1}{L_2}+\frac{17}{6}\Bigr)\varepsilon,\\
    &\nabla^2\mathcal{F}(x_{\widetilde T(\varepsilon)+1})
    \succcurlyeq
    -\Bigl[\frac{\sqrt{2}\,L_1}{\sqrt{L_2}}+\frac{13}{6}\sqrt{L_2}
    +\frac{L_H(1+\kappa)}{\sqrt{L_2}}\Bigr]\sqrt{\varepsilon}\,I,
  \end{aligned}
\end{equation}
and furthermore,
\begin{equation}\label{eq:iter-bound}
T(\varepsilon)\ \leqslant\ \widetilde T(\varepsilon)\ \leqslant\
1+\frac{24\sqrt{L_2}}{5}\bigl(\mathcal{F}(x_1)-\mathcal{F}_{\inf}\bigr)\varepsilon^{-3/2}.
\end{equation}
\end{thm}
\begin{proof}
We first prove the stationarity bounds in \eqref{eq:sosp-bounds}.
By definition of $\widetilde T(\varepsilon)$, we have
$|v_{\widetilde T(\varepsilon)}|>\sqrt{1/(1+\Lambda^2)}$.
Applying Lemma~\ref{lem2.6} under this condition yields
\begin{equation*}
\|\nabla \mathcal{F}(x_{\widetilde T(\varepsilon)+1})\|
\leqslant 2(L_1+\alpha)\Lambda^3+\frac{L_2}{2}\Lambda^2+(\varepsilon_2+\alpha)\Lambda+3\varepsilon_1.
\end{equation*}
Using the parameter values specified in Theorem~\ref{the.1}, we substitute into the gradient bound to obtain
\begin{align*}
\|\nabla \mathcal{F}(x_{\widetilde T(\varepsilon)+1})\|
&\leqslant
2\bigl(L_1+\sqrt{L_2\varepsilon}\bigr)\Bigl(\frac{\varepsilon}{L_2}\Bigr)^{3/2}
+\frac{L_2}{2}\Bigl(\frac{\varepsilon}{L_2}\Bigr)
+\Bigl(\frac{\sqrt{L_2\varepsilon}}{12}+\sqrt{L_2\varepsilon}\Bigr)\sqrt{\frac{\varepsilon}{L_2}}
+\frac{\varepsilon}{4} \\
&=
\frac{2L_1}{L_2^{3/2}}\varepsilon^{3/2}
+\frac{2}{L_2}\varepsilon^{2}
+\frac{11}{6}\varepsilon.
\end{align*}
By $0<\varepsilon\leqslant \min\{L_2/2,\,1\}$, we can easily get
\[
\|\nabla \mathcal{F}(x_{\widetilde T(\varepsilon)+1})\|
\leqslant \Bigl(\frac{\sqrt{2}\,L_1}{L_2}+\frac{17}{6}\Bigr)\varepsilon.
\]
Moreover, Lemma~\ref{lem2.6} provides the Hessian lower bound
\begin{equation*}
\nabla^2\mathcal{F}(x_{\widetilde T(\varepsilon)+1})
\succcurlyeq
-\Bigl\{2(L_1+\alpha)\Lambda^2+\alpha+\varepsilon_2
+L_H\bigl[(1+\kappa)\Lambda+2A\bigr]\Bigr\}I,
\end{equation*}
with $A=\min\{\varepsilon_1/\ell_1,\ \varepsilon_2/(2L_H)\}$.
Since $A\leqslant \varepsilon_2/(2L_H)$, we have
$L_H[(1+\kappa)\Lambda+2A]\leqslant L_H(1+\kappa)\Lambda+\varepsilon_2$.
Substituting the parameter choices from Theorem~\ref{the.1} into the above bound gives
\begin{align*}
&~2(L_1+\alpha)\Lambda^2+\alpha+\varepsilon_2
+L_H\bigl[(1+\kappa)\Lambda+2A\bigr] \\
\leqslant
&~2\bigl(L_1+\sqrt{L_2\varepsilon}\bigr)\frac{\varepsilon}{L_2}
+\sqrt{L_2\varepsilon}
+2\cdot\frac{\sqrt{L_2\varepsilon}}{12}
+\frac{L_H(1+\kappa)}{\sqrt{L_2}}\sqrt{\varepsilon}.
\end{align*}
Since $0<\varepsilon\leqslant \min\{L_2/2,\,1\}$, the right-hand side of the previous inequality can be simplified, yielding
\[
\nabla^2\mathcal{F}(x_{\widetilde T(\varepsilon)+1})
\succcurlyeq
-\Bigl[\frac{\sqrt{2}\,L_1}{\sqrt{L_2}}+\frac{13}{6}\sqrt{L_2}
+\frac{L_H(1+\kappa)}{\sqrt{L_2}}\Bigr]\sqrt{\varepsilon}\,I,
\]
which establishes the Hessian bound in \eqref{eq:sosp-bounds}. Consequently, we have
$T(\varepsilon)\leqslant \widetilde T(\varepsilon)$.

We now proceed to bound $\widetilde T(\varepsilon)$.  By the definition of $\widetilde T(\varepsilon)$, for any $t<\widetilde T(\varepsilon)$, we have
$|v_t|\leqslant \sqrt{1/(1+\Lambda^2)}$.
Hence Lemma~\ref{lem:step-decrease} is applicable and gives
\[
\mathcal{F}(x_{t+1})-\mathcal{F}(x_t)\leqslant \Lambda\varepsilon_1+(\varepsilon_2-\alpha)\frac{\Lambda^2}{2}
+\frac{L_2}{6}\Lambda^3.
\]
Substituting the parameter choices from Theorem~\ref{the.1} into this inequality leads to the per‑iteration decrease
\begin{equation}\label{eq:per-dec}
\mathcal{F}(x_{t+1})-\mathcal{F}(x_t)\leqslant -\frac{5}{24}\,\frac{\varepsilon^{3/2}}{\sqrt{L_2}},
\qquad \forall\, t<\widetilde T(\varepsilon).
\end{equation}
Summing \eqref{eq:per-dec} over $t=1,\ldots,\widetilde T(\varepsilon)-1$ and noting that the total possible decrease in $\mathcal{F}$ is at most $\mathcal{F}(x_1)-\mathcal{F}_{\inf}$, we obtain
\[
\mathcal{F}(x_1)-\mathcal{F}_{\inf}\geqslant \sum_{t=1}^{\widetilde T(\varepsilon)-1}\bigl(\mathcal{F}(x_t)-\mathcal{F}(x_{t+1})\bigr)
\geqslant (\widetilde T(\varepsilon)-1)\,\frac{5}{24}\,\frac{\varepsilon^{3/2}}{\sqrt{L_2}}.
\]
Rearranging gives the desired upper bound on $\widetilde T(\varepsilon)$ in \eqref{eq:iter-bound}.
Together with the already proved relation $T(\varepsilon)\leqslant \widetilde T(\varepsilon)$, the proof is complete.
\end{proof}
\begin{rem}
The bound \eqref{eq:iter-bound} in Theorem \ref{the.1} shows that HSDA attains an $\mathcal{O}(\varepsilon,\sqrt{\varepsilon})$-second-order stationary point within $\mathcal{O}(\varepsilon^{-3/2})$ outer iterations. This iteration complexity matches the best known rates for second-order methods in nonconvex-strongly concave minimax optimization; see, for example, \cite{luo2022finding,wang2025gradient}.	
\end{rem}

\section{Inexact Homogeneous Second-Order Descent Ascent Algorithm}\label{sec_ihsda}
The complexity analysis in Section~\ref{HSDA_al} assumes that the homogenized eigenvalue subproblem~\eqref{eq:subprob} is solved exactly at every outer iteration. In practice, however, this assumption can be prohibitive: for large-scale problems, computing the smallest eigenpair of the homogenized matrix $G_t(\alpha)$ typically requires expensive matrix factorizations or many iterations of a Krylov-type eigensolver. To overcome this limitation, we propose in this section an inexact homogeneous second-order descent ascent (IHSDA) algorithm, which solves the homogenized subproblem only approximately via a Lanczos procedure with carefully controlled residual. We prove that IHSDA retains the single‑loop structure of HSDA and achieves the same outer-iteration complexity.

Unlike the exact HSDA method, IHSDA avoids solving the homogenized eigenvalue subproblem \eqref{eq:subprob} exactly. It instead employs a Lanczos procedure to obtain an approximate solution, which comprises two main steps:
\begin{itemize}
  \item \textbf{Inexact Solution via Lanczos Iteration}. The homogenized subproblem \eqref{eq:subprob} is solved approximately using the Lanczos Method with Skewed Randomization \cite[Algorithm~4]{zhang2025homogeneous}. This yields a Ritz pair $(-\zeta_t,[\hat{u}_t;\hat{v}_t])$  of $G_t(\alpha_t)$ with Ritz residual $[k_t;\varrho_t]$,  satisfying 
  \begin{equation}\label{eq:Ritz-residual}
  	G_t(\alpha_t)
  	\begin{bmatrix}
  		\hat{u}_t\\ \hat{v}_t
  	\end{bmatrix}
  	+\zeta_t
  	\begin{bmatrix}
  		\hat{u}_t\\ \hat{v}_t
  	\end{bmatrix}
  	=
  	\begin{bmatrix}
  		k_t\\ \varrho_t
  	\end{bmatrix},
  	\ 
  	|\delta_t-\zeta_t|\leqslant e_t,
  	\ 
  	k_t^{\top}\hat{u}_t + \varrho_t \hat{v}_t = 0, \ \big\|[\hat{u}_t; \hat{v}_t]\big\| = 1,
  \end{equation}
where $-\delta_t:=\lambda_1\!\big(G_t(\alpha_t)\big)$ is the true smallest eigenvalue and the approximation accuracy satisfies $|\delta_t-\zeta_t|\leqslant e_t$  for a prescribed tolerance $e_t$ defined later.
  \item \textbf{Direction Generation and Safeguard.} The search direction $s_t$ is generated from the approximate eigenvector $[\hat{u}_t;\hat{v}_t]$ using the same classification rule as in the exact HSDA algorithm (see \eqref{eq:dk-classification-clean}).
  If $|\hat{v}_t|>\sqrt{1/(1+\Lambda^{2})}$, and the Lanczos residual in \eqref{eq:Ritz-residual} is sufficiently small, the next iterate $x_{t+1}$ can be certified as an
  $\mathcal{O}(\varepsilon,\sqrt{\varepsilon})$-second-order stationary point;
  Otherwise, a large residual triggers an increase of the parameter  $\alpha$ and a re-computation of the Ritz pair to obtain a more accurate approximation of the smallest eigenpair (see Lemma~\ref{lem6.6} for details).
 \end{itemize}
By substituting Step 3 of the HSDA algorithm with the inexact Lanczos procedure described above, we obtain the complete IHSDA algorithm for solving problem \eqref{eq:minimax}. The detailed algorithm is formally stated in Algorithm~\ref{alg:IHSDA}.
\begin{algorithm}
	\caption{Inexact Homogeneous Second-Order Descent Ascent (IHSDA) Algorithm}
	\label{alg:IHSDA}
	\begin{algorithmic}
		\STATE{\textbf{Step 1}: \textbf{Input} $x_1$, $y_0$, $\eta_1>0$, $\eta_2>0$, $L_1>0$, $L_2>0$, $B_g>0$,
			$\omega\in(1/4,1/2)$, $\{N_t\geqslant 1\}$, $\varepsilon>0$, $\Lambda>0$, and set $t=1$.}
		\STATE{\textbf{Step 2}: Update $y_t$:}
		\STATE{\quad\textbf{(2a)}: Set $i=0$, $y_i^{t}= \tilde y_i^{t}=y_{t-1}$.}
		\STATE{\quad\textbf{(2b)}: Update $y_i^{t}$ and $\tilde y_i^{t}$:
			\begin{align*}
				y_{i+1}^{t} &= \tilde y_i^{t} + \eta_{1} \nabla_y f\big(x_t,\tilde y_i^{t}\big),\\
				\tilde y_{i+1}^{t} &= y_{i+1}^{t} + \eta_{2}\big(y_{i+1}^{t}-y_i^{t}\big).
		\end{align*}}
		\STATE{\quad\textbf{(2c)}: If $i \geqslant N_t-1$, set $y_t = y_{N_t}^{t}$ and go to Step 3; otherwise set $i = i+1$ and go to Step (2b).}	
		\STATE{\textbf{Step 3}: Compute
			\begin{equation*}
				g_t = \nabla_x f(x_t,y_t),\qquad
				H_t = \big[\nabla^2_{xx}f - \nabla^2_{xy}f(\nabla^2_{yy}f)^{-1}\nabla^2_{yx}f\big](x_t,y_t).
			\end{equation*}
			Set $e_t=\sqrt{L_2\varepsilon}$ and $\alpha_t=\sqrt{L_2\varepsilon}$, and $G_t(\alpha_t):=\begin{bmatrix} H_t & g_t\\ g_t^{\top} & -\alpha_t\end{bmatrix}$.}
		\STATE{\qquad \textbf{(3a)}    By applying \cite[Algorithm~4]{zhang2025homogeneous} to compute a Ritz pair of $G_t(\alpha_{t})$, i.e., $(-\zeta_t,[\hat{u}_t;\hat{v}_t])$  with Ritz residual $[k_t;\varrho_t]$, which satisfies
			\begin{equation}
				G_t(\alpha_t)
				\begin{bmatrix}
					\hat{u}_t\\ \hat{v}_t
				\end{bmatrix}
				+\zeta_t
				\begin{bmatrix}
					\hat{u}_t\\ \hat{v}_t
				\end{bmatrix}
				=
				\begin{bmatrix}
					k_t\\ \varrho_t
				\end{bmatrix},
				\ 
				|\delta_t-\zeta_t|\leqslant e_t,
				\ 
				k_t^{\top}\hat{u}_t + \varrho_t \hat{v}_t = 0, \ \big\|[\hat{u}_t; \hat{v}_t]\big\| = 1.
			\end{equation}
			If $|\hat v_t|\leqslant \sqrt{1/(1+\Lambda^2)}$, go to Step 4;}
		\STATE{\qquad \textbf{(3b)}  If $\|k_t\|\leqslant \varepsilon/2$, set $x_{t+1}=x_t+\dfrac{\hat u_t}{\hat v_t}$ and \textbf{terminate}. Otherwise, set
			\begin{equation*}
				\alpha_{t}
				= 3\sqrt{L_2\varepsilon}
				+ 2\|g_t\|\Lambda
				+ (L_1+\zeta_t)\Lambda^2,
				\qquad
				e_t
				= \min\!\left\{\frac{\varepsilon}{4},\ \frac{\sqrt{L_2}\,\varepsilon^{5/2}}{64\,(L_1+\alpha_t+B_g)^2}\right\},
			\end{equation*}
			and go to Step~3(a).}
		\STATE{\textbf{Step 4}: Update the direction $s_t$:
			\begin{equation*}
				s_t =
				\begin{cases}
					\dfrac{\hat u_t}{\hat v_t}, & \qquad |\hat v_t|\geqslant \omega ,
					\\[9pt]
					\operatorname{sgn}\!\big(-g_t^{\top}\hat u_t\big)\hat u_t, & \qquad |\hat v_t|<\omega .
				\end{cases}
			\end{equation*}
		}
		
		\STATE{\textbf{Step 5}: Compute $\tau_t=\Lambda/\|s_t\|$, update $x_{t+1} = x_t + \tau_t s_t$, set $t=t+1$, and go to Step~2.}
	\end{algorithmic}
\end{algorithm}

In the following subsection, we prove that IHSDA finds an $\mathcal{O}(\varepsilon,\sqrt{\varepsilon})$-second-order stationary point of $\mathcal{F}(x)$ for problem~\eqref{eq:minimax} with an outer iteration complexity of $\mathcal{O}(\varepsilon^{-3/2})$. Furthermore, we derive a high-probability upper bound on total number of Hessian-vector products  is $\tilde{\mathcal{O}}\big(\varepsilon^{-7/4}\big)$.

\subsection{Complexity Analysis}
For our subsequent analysis, we adopt the following standard assumption commonly used in the complexity analysis of second-order methods~\cite{Cartis2011ARC,Royer2018ComplexityAO}.
\begin{ass}\label{ass:grad-bounded}
  There exists a constant $B_g > 0$, independent of $t$, such that
\begin{equation*}
  \|g(x_t,y_t)\| \leqslant B_g, \qquad \forall\, t \geqslant 1.
\end{equation*}
\end{ass}
We now proceed to derive a quantitative decrease bound for the value function \eqref{eq:value} under the condition $|\hat{v}_t|\leqslant \sqrt{1/(1+\Lambda^{2})}$.
\begin{lem}\label{cor6.1}
Under Assumption~\ref{as.1}, let $\omega\in(1/4,1/2)$ and $\Lambda \leqslant \sqrt{2}/2$. Then for any $\varepsilon>0$, and whenever
$|\hat{v}_t|\leqslant \sqrt{1/(1+\Lambda^{2})}$ the following decrease bound holds with probability at least $1-4p$ (where $p\in(\exp(-n),1)$):
\begin{equation*}
  \mathcal{F}(x_{t+1})-\mathcal{F}(x_t)
  \leqslant 4|\varrho_t| 
   - \frac{\alpha_{t}}{2} \Lambda^2
  + \Lambda \varepsilon_1
  + \frac{\Lambda^{2}}{2} \varepsilon_2
  + \frac{L_{2}}{6} \Lambda^{3}.
\end{equation*}
\end{lem}
\begin{proof}
Let $E_t := \tau_t g_t^\top s_t+\frac{\tau_t^2}{2} s_t^\top H_t s_t$.
As in the analysis of HSDA, we have $\tau_t\in(0,1]$ and
\begin{equation}\label{inexact_decrease1}
  \mathcal{F}(x_{t+1})-\mathcal{F}(x_t) \leqslant E_t + \Lambda \varepsilon_1
   + \frac{\Lambda^{2}}{2} \varepsilon_2
   + \frac{L_{2}}{6} \Lambda^{3}.
\end{equation}
{\bfseries Case 1}: $|\hat{v}_t|\leqslant \omega$. From \eqref{eq:Ritz-residual} we obtain
\begin{equation*}
\hat{u}_t^\top H_t \hat{u}_t = k_t^\top \hat{u}_t - \zeta_t \|\hat{u}_t\|^2 - \hat{v}_t g_t^\top \hat{u}_t,
\quad
g_t^\top \hat{u}_t = \varrho_t + \hat{v}_t (\alpha_{t}-\zeta_t).
\end{equation*}
Since $s_t=\operatorname{sgn}(-g_t^\top \hat{u}_t)\hat{u}_t$, we have $\|s_t\|=\|\hat{u}_t\|$ and with $\tau_t\|s_t\|=\Lambda$, also $\tau_t\|\hat{u}_t\|=\Lambda$. Therefore,
\begin{align}\label{inexact_decrease2}
E_t
&= \tau_t g_t^\top s_t + \frac{\tau_t^2}{2} s_t^\top H_t s_t \nonumber\\
&= \tau_t \operatorname{sgn}(-g_t^\top \hat{u}_t) g_t^\top \hat{u}_t
   + \frac{1}{2}\tau_t^2 \hat{u}_t^\top H_t \hat{u}_t \nonumber\\
&= -\tau_t |g_t^\top \hat{u}_t|
   + \frac{1}{2}\tau_t^2 k_t^\top \hat{u}_t
   - \frac{1}{2}\tau_t^2 \hat{v}_t g_t^\top \hat{u}_t
   - \frac{1}{2}\tau_t^2 \zeta_t \|\hat{u}_t\|^2 \nonumber\\
&\leqslant -\tau_t |g_t^\top \hat{u}_t|
   + \frac{1}{2}\tau_t^2 k_t^\top \hat{u}_t
   + \frac{1}{2}\tau_t^2 |\hat{v}_t| |g_t^\top \hat{u}_t|
   - \frac{1}{2}\tau_t^2 \zeta_t \|\hat{u}_t\|^2 \nonumber\\
&= - \frac{1}{2}\tau_t^2 \hat{v}_t \varrho_t
   - \Bigl(\tau_t - \frac{1}{2}\tau_t^2 |\hat{v}_t|\Bigr) |g_t^\top \hat{u}_t|
   - \frac{\zeta_t}{2} \Lambda^2.
\end{align}
Since $\tau_t \leqslant 1$ and $|\hat{v}_t|\leqslant \omega<1$, we have $\tau_t^2 |\hat{v}_t|\leqslant \tau_t\leqslant1$.
Combining this with \eqref{inexact_decrease1} and \eqref{inexact_decrease2} yields
\begin{equation}\label{eq:large-case-b}
\mathcal{F}(x_{t+1})-\mathcal{F}(x_t)
\leqslant |\varrho_t|
- \frac{\zeta_t}{2} \Lambda^2
+ \Lambda \varepsilon_1
+ \frac{\Lambda^{2}}{2} \varepsilon_2
+ \frac{L_{2}}{6} \Lambda^{3}.
\end{equation}
{\bfseries Case 2}: $\omega \leqslant |\hat{v}_t| \leqslant \sqrt{1/(1+\Lambda^2)}$.  Here $s_t=\hat{u}_t/\hat{v}_t$.  Using \eqref{eq:Ritz-residual} we obtain
\begin{equation*}
s_t^\top H_t s_t + g_t^\top s_t
= -\zeta_t \|s_t\|^2 + \frac{k_t^\top \hat{u}_t}{\hat{v}^{2}_t},
\quad
g_t^\top s_t = -\zeta_t + \alpha_{t} + \frac{\varrho_t}{\hat{v}_t}.
\end{equation*}
From the orthogonality relation $k_t^{\top}\hat{u}_t + \varrho_t \hat{v}_t = 0$ in \eqref{eq:Ritz-residual}, it follows that
\begin{align*}
E_t
&= \tau_t g_t^\top s_t + \frac{\tau_t^2}{2} s_t^\top H_t s_t \nonumber\\
&= \tau_t g_t^\top s_t
  + \frac{1}{2}\tau_t^2 \left(\frac{k_t^\top \hat{u}_t}{\hat{v}_t^2} - g_t^\top s_t - \zeta_t \|s_t\|^2 \right) \nonumber\\
&= \left(\tau_t-\frac{1}{2}\tau_t^2\right)\!\left(\frac{\varrho_t}{\hat{v}_t} + \alpha_{t} - \zeta_t\right)
   + \frac{\tau_t^2}{2}\!\left(\frac{k_t^\top \hat{u}_t}{\hat{v}_t^2}\right)
   - \frac{\zeta_t}{2} \Lambda^2 \nonumber\\
&= \left(\tau_t-\frac{1}{2}\tau_t^2\right)(\alpha_{t}-\zeta_t)
   - (\tau_t^2-\tau_t)\frac{\varrho_t}{\hat{v}_t}
   - \frac{\zeta_t}{2} \Lambda^2.
\end{align*}
Since $\tau_t\in(0,1]$ and $|\hat{v}_t|\geqslant \omega \geqslant 1/4$, 
\begin{equation*}
-(\tau_t^2-\tau_t)\frac{\varrho_t}{\hat{v}_t}
\leqslant \left|\frac{\varrho_t}{\omega}\right|
\leqslant 4|\varrho_t|.
\end{equation*}
Hence,
\begin{equation}\label{inexact_decrease3}
  E_t \leqslant \left(\tau_t-\frac{1}{2}\tau_t^2\right)(\alpha_{t}-\zeta_t) + 4|\varrho_t| - \frac{\zeta_t}{2} \Lambda^2.
\end{equation}
Moreover, by Theorems 5 and 6 of~\cite{zhang2025homogeneous}, with probability at least $1-4p$ we have
\begin{equation}\label{inexact_decrease4}
\zeta_t \geqslant \alpha_{t}.
\end{equation}
Substituting \eqref{inexact_decrease3} and \eqref{inexact_decrease4} into \eqref{inexact_decrease1} gives
\begin{equation}\label{inexact_decrease5}
\mathcal{F}(x_{t+1})-\mathcal{F}(x_t)
\leqslant 4|\varrho_t|
- \frac{\zeta_t}{2} \Lambda^2
+ \Lambda \varepsilon_1
+ \frac{\Lambda^{2}}{2} \varepsilon_2
+ \frac{L_{2}}{6} \Lambda^{3}.
\end{equation}
Since \eqref{eq:large-case-b} provides a tighter (i.e., smaller) upper bound, we unify the analysis of both cases by adopting \eqref{inexact_decrease5} as a common estimate. 
Using $\zeta_t \geqslant \alpha_{t}$ from \eqref{inexact_decrease4} (which holds with the stated probability), we obtain
\begin{align*}
\mathcal{F}(x_{t+1})-\mathcal{F}(x_t)
&\leqslant 4|\varrho_t|
   - \frac{\zeta_t}{2} \Lambda^2
   + \Lambda \varepsilon_1
   + \frac{\Lambda^{2}}{2} \varepsilon_2
   + \frac{L_{2}}{6} \Lambda^{3} \\
&\leqslant 4|\varrho_t| 
   - \frac{\alpha_{t}}{2} \Lambda^2
   + \Lambda \varepsilon_1
   + \frac{\Lambda^{2}}{2} \varepsilon_2
   + \frac{L_{2}}{6} \Lambda^{3},
\end{align*}
which completes the proof.
\end{proof}

We now consider the case $|\hat{v}_t| > \sqrt{1/(1+\Lambda^{2})}$. 
The following lemma demonstrates that, under appropriate parameter choices, one of two outcomes must occur with high probability: either next iterate is already an $\mathcal{O}(\varepsilon,\sqrt{\varepsilon})$-second-order stationary point, or, after increasing the parameter $\alpha_{t}$ and re-solving the homogenized eigenvalue subproblem~\eqref{eq:subprob}—the Ritz residual will become sufficiently small. 
\begin{lem}\label{lem6.6}
Under Assumptions \ref{as.1} and \ref{ass:grad-bounded}, consider the case $|\hat{v}_t| > \sqrt{1/(1+\Lambda^2)}$.
Set $\Lambda=\sqrt{\varepsilon/L_{2}}$, $\varepsilon_1=\varepsilon/12$ and $\varepsilon_2=\sqrt{L_2\varepsilon}/12$, 
with $\varepsilon \leqslant \min\Bigl\{{L_{2}^3}/{36}, L_{2}/2, 1\Bigr\}.$
Then the following holds:
\begin{enumerate}[label=(\arabic*)]
\item If the Ritz residual satisfies $\|k_t\|\leqslant \varepsilon/2$, then the next iterate point $x_{t+1}$ is an $\mathcal{O}(\varepsilon,\sqrt{\varepsilon})$-second-order stationary point.
\item Otherwise, whenever $|\hat{v}_t|>\sqrt{1/(1+\Lambda^2)}$, with probability at least $1-4p$, 
we have  $\|k_t\| \leqslant \varepsilon/2.$
\end{enumerate}
\end{lem}
\begin{proof}
{\bfseries Proof of (1)}. Without loss of generality, we can assume $\hat v_t>0$, as the sign of the approximate eigenvector
$[\hat u_t;\hat v_t]$ can be flipped without affecting any subsequent derivations.
We show that when $\|k_t\|\leqslant \varepsilon/2$, 
$\lambda_1\!\big(\nabla^2 \mathcal{F}(x_{t+1})\big) \geqslant \mathcal{O}(-\sqrt{\varepsilon})$
and
$\|\nabla \mathcal{F}(x_{t+1})\| \leqslant \mathcal{O}(\varepsilon)$.
From the Ritz condition \eqref{eq:Ritz-residual} we obtain
\begin{equation}\label{eq:ritz-quadratic}
    -\zeta_t
    = -\alpha_{t} \hat{v}_t^2 + 2 \hat{v}_t g_t^\top \hat{u}_t + \hat{u}_t^\top H_t \hat{u}_t .
\end{equation}
Using $\|\hat{u}_t\|^2 + \hat{v}_t^2 = 1$, \eqref{eq:ritz-quadratic} can be rewritten as
\begin{align}
(\zeta_t - \alpha_{t}) \hat{v}_t^2
&= -2 \hat{v}_t g_t^\top \hat{u}_t
   - \Bigl( \zeta_t + \frac{\hat{u}_t^\top H_t \hat{u}_t}{\|\hat{u}_t\|^2} \Bigr) \|\hat{u}_t\|^2 \nonumber\\
&\leqslant 2 \hat{v}_t \sqrt{1-\hat{v}_t^2} \|g_t\|
   - \Bigl( \zeta_t + \frac{\hat{u}_t^\top H_t \hat{u}_t}{\|\hat{u}_t\|^2} \Bigr) \|\hat{u}_t\|^2 \nonumber\\
&\leqslant 2 \hat{v}_t \sqrt{1-\hat{v}_t^2} \|g_t\|
   - \bigl(\zeta_t + \lambda_1(H_t)\bigr) (1-\hat{v}_t^2).\label{eq:zeta-alpha-relation}
\end{align}
Moreover, using $\Lambda \geqslant \sqrt{1-\hat{v}_t^2}/\hat{v}_t$ and
$\lambda_1(H_t) \leqslant L_1$, \eqref{eq:zeta-alpha-relation} yields
\begin{equation}\label{eq:zeta-alpha-upper}
  \zeta_t - \alpha_{t}
  \leqslant 2\Lambda \|g_t\|
  + \bigl|\lambda_1(H_t)+\zeta_t\bigr|\Lambda^2
  \leqslant 2\Lambda \|g_t\|
  + (L_1+\zeta_t)\Lambda^2 .
\end{equation}
Since $H_t+\delta_t I \succcurlyeq 0$ and
$\delta_t \leqslant \zeta_t + e_t \leqslant \zeta_t + \alpha_{t}$, we have
$\lambda_1(H_t) + \delta_t \geqslant 0$.  Combining this with \eqref{eq:zeta-alpha-upper} gives
\begin{equation}\label{lem:6.6.1}
  \lambda_1(H_t)
  + 2\alpha_{t} + 2\|g_t\|\Lambda + (L_1+\zeta_t)\Lambda^2
  \geqslant 0 .
\end{equation}
When $\|k_t\|\leqslant \varepsilon/2$, the argument parallels that of Lemma \ref{lem2.6}: we have 
$\|s_t\| = \|\hat{u}_t/\hat{v}_t\|\leqslant \Lambda$ and
\begin{equation}\label{add:1}
        H_{t+1} \succcurlyeq - \big(2\alpha_{t} + 2\|g_t\|\Lambda + (L_1+\zeta_t)\Lambda^2 +  L_H\left[(1+\kappa)\Lambda+2A\right]\big) I .
\end{equation}
A simple norm estimate yields
\begin{align}
\zeta_t \leqslant \|G_t(\alpha_t)\|
&\leqslant \max_{\|[u;v]\|=1}\left\lvert\!
\begin{bmatrix} u \\ v \end{bmatrix}^{\!\top}
\begin{bmatrix} H_t & 0 \\[2pt] 0 & -\alpha_{t} \end{bmatrix}
\begin{bmatrix} u \\ v \end{bmatrix}\!\right\rvert
+ \max_{\|[u;v]\|=1}\left\lvert\!
\begin{bmatrix} u \\ v \end{bmatrix}^{\!\top}
\begin{bmatrix} 0 & g_t \\[2pt] g_t^\top & 0 \end{bmatrix}
\begin{bmatrix} u \\ v \end{bmatrix} \!\right\rvert\nonumber\\
&\leqslant \max\{L_1,\alpha_{t}\} + \|g_t\|
 \leqslant L_1 + \alpha_{t} + B_g .\label{normbound}
\end{align}
Inserting \eqref{normbound} and the parameter choices into \eqref{add:1}, then applying Lemma \ref{lem.2} (which controls the Hessian approximation error), we obtain after elementary simplifications 
\begin{align*}
\nabla^2 \mathcal{F}(x_{t+1})
&\succcurlyeq
- \big(2\alpha_{t} + 2\|g_t\|\Lambda + (L_1+\zeta_t)\Lambda^2 + \varepsilon_{2}
+ L_H\left[(1+\kappa)\Lambda+2A\right]\big) I \\[0.2em]
&\succcurlyeq
- \big(2\alpha_{t} + 2B_g\Lambda + (2L_1+\alpha_{t}+B_g)\Lambda^2 + 2\varepsilon_{2}
+ L_H(1+\kappa)\Lambda\big) I \\[0.2em]
&=
- \Bigl(\Bigl(\frac{13}{6}\sqrt{L_2} + \frac{2B_g+L_H(1+\kappa)}{\sqrt{L_2}}\Bigr)\sqrt{\varepsilon}
      + \frac{2L_1+B_g}{L_2}\varepsilon
      + \frac{1}{\sqrt{L_2}}\varepsilon^{3/2}\Bigr) I \\[0.2em]
&\succcurlyeq
- \Bigl(\Bigl(\frac{13}{6}\sqrt{L_2} + \frac{2B_g+L_H(1+\kappa)}{\sqrt{L_2}}\Bigr)\sqrt{\varepsilon}
      + \frac{2L_1+B_g}{\sqrt{L_2}}\sqrt{\varepsilon}
      + \sqrt{L_2}\sqrt{\varepsilon}\Bigr) I \\[0.2em]
&=
- \Bigl(\frac{19}{6}\sqrt{L_2}
        + \frac{2L_1+3B_g+L_H(1+\kappa)}{\sqrt{L_2}}\Bigr)\sqrt{\varepsilon}\, I .
\end{align*}
We now bound the gradient norm $\|\nabla \mathcal{F}(x_{t+1})\|$. Using the second‑order Lipschitz continuity of $\nabla^2 \mathcal{F}$, together with \eqref{gk.gap} and
\eqref{eq:Ritz-residual}, we have
\begin{equation}\label{eq:gtplus1-bound}
\begin{aligned}
\|g_{t+1}\|
&\leqslant \|g_{t+1}-g_t - H_t s_t\| + \|g_t + H_t s_t\| \\
&= \|g_{t+1}-g_t - H_t s_t\| + \|k_t/\hat{v}_t - \zeta_t s_t\| \\
&\leqslant \frac{L_2}{2}\Lambda^2 + \varepsilon_2 \Lambda + 2\varepsilon_1
       + \frac{\|k_t\|}{\omega} + |\zeta_t|\Lambda .
\end{aligned}
\end{equation}
Theorem 6 of \cite{zhang2025homogeneous} guarantees that with probability at least $1-4p$, 
$|\varrho_t|\leqslant \varepsilon^2/(16 L_{2}^2)$.
Moreover, \eqref{eq:Ritz-residual} implies the scalar identity
$\zeta_t = \alpha_{t} + {\varrho_t}/{\hat{v}_t} - g_t^\top s_t .$
In addition, in the regime $|\hat{v}_t| > \sqrt{1/(1+\Lambda^2)}$, we have
${1}/{|\hat{v}_t|} < \sqrt{1+\Lambda^2} = \sqrt{1+\varepsilon/L_2} \leqslant \sqrt{2}$,
where we used $0<\varepsilon\leqslant L_2$. 
Hence
\begin{equation*}
  \Big|\frac{\varrho_t}{\hat{v}_t}\Big|
  \leqslant \frac{\varepsilon^2}{16 L_{2}^2}\cdot \frac{1}{|\hat{v}_t|}
  \leqslant \frac{\sqrt{2}\varepsilon^2}{16L_2^2}.
\end{equation*}
Combining this bound with the expression for $\zeta_t$ yields
\begin{equation}\label{eq:zeta-bound}
|\zeta_t|
\leqslant |\alpha_{t}| + \Big|\frac{\varrho_t}{\hat{v}_t}\Big| + |g_t^\top s_t| 
\leqslant \sqrt{L_2\varepsilon}
       + \frac{\sqrt{2}\varepsilon^2}{16L_2^2}
       + B_g \Lambda .
\end{equation}
Substituting \eqref{eq:zeta-bound} and the parameter values into \eqref{eq:gtplus1-bound}, and again invoking Lemma \ref{lem.2} for the gradient error, we arrive at 
\begin{align*}
\|\nabla \mathcal{F}(x_{t+1})\|
&\leqslant \frac{L_2}{2}\Lambda^2 + \varepsilon_2 \Lambda + 3\varepsilon_1
   + \frac{\|k_t\|}{\omega} + |\zeta_t| \Lambda \\[0.1em]
&= \frac{5}{6}\varepsilon + \frac{\|k_t\|}{\omega} + |\zeta_t|\Lambda \\[0.1em]
&\leqslant \frac{5}{6}\varepsilon + \frac{\varepsilon}{2\omega}
   + \Bigl(\sqrt{L_2\varepsilon} + \frac{\sqrt2}{16L_2^2}\varepsilon^2 + B_g\Lambda\Bigr)\Lambda \\[0.2em]
&= \Bigl(\frac{11}{6} + \frac{1}{2\omega} + \frac{B_g}{L_2}\Bigr)\varepsilon
   + \frac{\sqrt2}{16L_2^{5/2}}\varepsilon^{5/2} \\[0.2em]
&\leqslant \Bigl(\frac{23}{6} + \frac{B_g}{L_2} + \frac{\sqrt2}{16L_2}\Bigr)\varepsilon.
\end{align*}
This shows that $x_{t+1}$ is already an $\mathcal{O}(\varepsilon,\sqrt{\varepsilon})$-second-order stationary point.

{\bfseries Proof of (2)}. We first show that
$\lambda_2(G_t(\alpha_t)) - \lambda_1(G_t(\alpha_t)) \geqslant \sqrt{L_{2}\varepsilon}.$
From~\eqref{lem:6.6.1} with the initial choice  $\alpha_{t} = \sqrt{L_{2}\varepsilon}$ we have
\begin{equation}\label{lem:6.6.2}
\lambda_1(H_t) + 2\sqrt{L_{2}\varepsilon} + 2\|g_t\|\Lambda + (L_1+\zeta_t)\Lambda^2 \geqslant 0 .
\end{equation}
After updating $\alpha_t$, i.e., 
$\alpha_t = 3\sqrt{L_{2}\varepsilon} + 2\|g_t\|\Lambda + (L_1+\zeta_t)\Lambda^2,$ 
the Cauchy interlacing theorem together with \eqref{lem:6.6.2} yields
\begin{equation}\label{eq:eig-gap}
 \lambda_2(G_t(\alpha_t)) - \lambda_1(G_t(\alpha_t))
\geqslant \lambda_1(H_t) + \alpha_{t}
\geqslant \sqrt{L_{2}\varepsilon}.
\end{equation}
By Lemma~13 of~\cite{zhang2025homogeneous}, we have 
\begin{equation*}
\|k_t\| \leqslant \phi_t e_t + 2\big(\max\{L_1,\alpha_{t}\} + \|g_t\|\big)\, \sqrt{\frac{e_t}{\lambda_2(G_t(\alpha_t)) - \lambda_1(G_t(\alpha_t))}},
\end{equation*}
where $e_t$  is the prescribed accuracy of the Lanczos run and $\phi_t\in[0,1]$.
Thus, using \eqref{eq:eig-gap}, we have
\begin{equation*}
\begin{aligned}
\|k_t\|
&\leqslant \phi_t e_t + 2\big(\max\{L_1,\alpha_{t}\} + \|g_t\|\big)\, \sqrt{\frac{e_t}{\lambda_2(G_t(\alpha_t)) - \lambda_1(G_t(\alpha_t))}} \\
&\leqslant \phi_t e_t + 2\bigl(L_1 + \alpha_{t} + B_g\bigr)\, \sqrt{\frac{e_t}{\sqrt{L_2\varepsilon}}} \\
&\leqslant \frac{\varepsilon}{2},
\end{aligned}
\end{equation*}
which completes the proof.
\end{proof}

We are now ready to present the main complexity result of this section: a high-probability bound on the iteration complexity of the IHSDA algorithm.  Let $\varepsilon>0$ be the target accuracy. Recall that $T(\varepsilon)$ denotes the first iteration at which an $\mathcal{O}(\varepsilon,\sqrt{\varepsilon})$-second-order stationary point is obtained. Additionally, we define a verifiable stopping index based on the approximate eigenvector component $\hat v_t$:
\begin{equation}\label{eq:That-eps-IHSDA}
\hat{T}(\varepsilon)
:= \min\Bigl\{t \,\Bigm|\,
|\hat v_t|>\sqrt{1/(1+\Lambda^2)}\ \text{and}\ \|k_t\|\leqslant \varepsilon/2
\Bigr\},
\end{equation}
where $k_t$ is as defined in \eqref{eq:Ritz-residual}.
By Lemma~\ref{lem6.6}, $\hat{T}(\varepsilon)$ is finite with high probability and satisfies $T(\varepsilon)\leqslant \hat{T}(\varepsilon)$.
\begin{thm}\label{thm:complexity_inexact_hsodm}
Under Assumptions \ref{as.1} and \ref{ass:grad-bounded}, define
\begin{equation*}
K_\varepsilon
:= 1 + 6\sqrt{L_{2}}\bigl(\mathcal{F}(x_1)-\mathcal{F}_{\inf}\bigr)\,\varepsilon^{-3/2}.
\end{equation*}
Then the outer-iteration counts of IHSDA satisfy
\begin{equation}\label{eq:iter-bound-IHSDA}
T(\varepsilon)\leqslant \hat{T}(\varepsilon)\leqslant K_\varepsilon,
\end{equation}
and with probability at least $(1-4p)^{2K_\varepsilon}$, the algorithm returns an
$\mathcal{O}(\varepsilon,\sqrt{\varepsilon})$-second-order stationary point.
\end{thm}
\begin{proof}
We first bound $\hat{T}(\varepsilon)$.
Fix any outer iteration $t<\hat{T}(\varepsilon)$. By Lemma~\ref{lem6.6}, the Ritz pair used in this iteration satisfies $|\hat{v}_t|\leqslant \sqrt{1/(1+\Lambda^{2})}$.
Recall that
\begin{equation*}
\Lambda = \sqrt{\varepsilon/L_2},\quad
\alpha_{t} = \sqrt{L_2\varepsilon},\quad
\varepsilon_1 = \varepsilon/12,\quad
\varepsilon_2 = \sqrt{L_2\varepsilon}/12,
\end{equation*}
with $0<\varepsilon \leqslant \min\{L_2^3/36,\,L_2/2,\,1\}$.
By Theorem 6 of \cite{zhang2025homogeneous}, with probability at least $1-4p$, we have $|\varrho_t|\leqslant \varepsilon^2/(16L_2^2)$. 
On this high‑probability event, Lemma~\ref{cor6.1} yields
\begin{equation*}
\mathcal{F}(x_{t+1})-\mathcal{F}(x_t)
\leqslant
4|\varrho_t|
- \frac{\alpha_{t}}{2} \Lambda^2
+ \Lambda \varepsilon_1
+ \frac{\Lambda^{2}}{2} \varepsilon_2
+ \frac{L_{2}}{6} \Lambda^{3}.
\end{equation*}
If the Ritz pair is computed with a larger parameter $\alpha_{t}$,  the right‑hand side can only become smaller, so the bound remains valid.
Substituting the parameter values gives
\begin{align*}
\mathcal{F}(x_{t+1})-\mathcal{F}(x_t)
&\leqslant \frac{\varepsilon^2}{4L_2^2}
   - \frac{1}{2}\sqrt{L_2\varepsilon}\,\frac{\varepsilon}{L_2}
   + \sqrt{\frac{\varepsilon}{L_2}}\cdot\frac{\varepsilon}{12}
   + \frac{\varepsilon}{2L_2}\cdot\frac{\sqrt{L_2\varepsilon}}{12}
   + \frac{L_2}{6}\Bigl(\frac{\varepsilon}{L_2}\Bigr)^{3/2} \\
&= -\frac{5}{24}\frac{\varepsilon^{3/2}}{\sqrt{L_2}}
   + \frac{\varepsilon^2}{4L_2^2}
\;\leqslant\;
-\frac{1}{6}\,\frac{\varepsilon^{3/2}}{\sqrt{L_2}},
\end{align*}
where the last inequality uses $\varepsilon\leqslant L_2^3/36$.

Summing this per-iteration decrease over $t=1,\ldots,\hat{T}(\varepsilon)-1$ and noting that the total possible decrease of $\mathcal{F}$
is at most $\mathcal{F}(x_1)-\mathcal{F}_{\inf}$, we obtain
\begin{equation*}
\mathcal{F}(x_1)-\mathcal{F}_{\inf}
\geqslant
\sum_{t=1}^{\hat{T}(\varepsilon)-1}\bigl(\mathcal{F}(x_t)-\mathcal{F}(x_{t+1})\bigr)
\geqslant
(\hat{T}(\varepsilon)-1)\,\frac{1}{6}\,\frac{\varepsilon^{3/2}}{\sqrt{L_2}},
\end{equation*}
and therefore
\begin{equation*}
\hat{T}(\varepsilon)
\leqslant
1 + 6\sqrt{L_2}\bigl(\mathcal{F}(x_1)-\mathcal{F}_{\inf}\bigr)\varepsilon^{-3/2}.
\end{equation*}
By the definition of $\hat{T}(\varepsilon)$ in \eqref{eq:That-eps-IHSDA} together with Lemma~\ref{lem6.6}, we have $T(\varepsilon)\leqslant \hat{T}(\varepsilon)$. Combining this with the bound above establishes \eqref{eq:iter-bound-IHSDA}.

We next establish the high-probability statement. Lemma~\ref{lem6.6} states that whenever 
 $|\hat{v}_t|>\sqrt{1/(1+\Lambda^2)}$, either (i) $\|k_t\|\leqslant \varepsilon/2$ already holds, or (ii) after at most one additional Ritz‑pair computation we obtain $\|k_t\|\leqslant \varepsilon/2$ with probability at least $1-4p$. Hence each outer iteration involves at most two calls to the Lanczos routine. Over at most $K_\varepsilon$ outer iterations, the total number of Lanczos invocations is at most $2K_\varepsilon$. Consequently, the probability that every Lanczos call succeeds is at least $(1-4p)^{2K_\varepsilon}$.
On this event, the definition of $\hat{T}(\varepsilon)$ guarantees that  at iteration
$t=\hat{T}(\varepsilon)$ we have $|\hat{v}_t|>\sqrt{1/(1+\Lambda^2)}$ and $\|k_t\|\leqslant \varepsilon/2$.
Lemma~\ref{lem6.6} then implies that the next iterate is an $\mathcal{O}(\varepsilon,\sqrt{\varepsilon})$-second-order stationary point, which completes the proof.
\end{proof}
\begin{rem}
We remark that, by Bernoulli's inequality, $(1-4p)^{2K_\varepsilon} \geqslant 1-8K_\varepsilon p$ whenever $p<1/4$. Since $p\in(\exp(-n),1)$,  this condition can be satisfied by taking $n$ sufficiently large, for instance, under the mild requirement $n\geqslant \mathcal{O}(-\log\varepsilon)$. Consequently, the high-probability guarantee ``with probability at least $(1-4p)^{2K_\varepsilon}$''  stated in the theorem can be presented equivalently as  ``with probability at least $1-8K_\varepsilon p$'' without losing information.
Moreover, combining the iteration bound \eqref{eq:iter-bound-IHSDA} with the per-call complexity estimates of \cite[Theorem 6 and Lemma 12]{zhang2025homogeneous} yields the total number of Hessian-vector products  bound
\begin{equation*}
  \widetilde{\mathcal{O}}\!\Bigl(L_{2}^{1/4}\bigl(\mathcal{F}(x_1)-\mathcal{F}_{\inf}\bigr)\varepsilon^{-7/4}
  \sqrt{\max\{L_1,\alpha_{t}\}+B_g}\Bigr),
\end{equation*}
where $\widetilde{\mathcal{O}}(\cdot)$ hides logarithmic factors in $n$, $p^{-1}$, and $\varepsilon^{-1}$.
\end{rem}

We now compare the computational effort required to solve the inner subproblems in gradient norm regularized second order methods and in IHSDA, highlighting the structural advantages of the homogeneous formulation.

Algorithms such as IGRTR and ILMNegCur \cite{wang2025gradient} for minimizing $\mathcal{F}(x)$ require solving a regularized Newton system 
\begin{equation}\label{eq:newton-perturbed}
	\bigl(H_t + \varepsilon_N I\bigr)d_t = -g_t,
\end{equation}
where the perturbation parameter $\varepsilon_N$ is chosen on the order of $\|g_t\|^{1/2}$. 
Computing an $\varepsilon$-accurate solution  of \eqref{eq:newton-perturbed}  takes at most 
\[
\mathcal{O}\!\left(
\sqrt{\kappa\bigl(H_t+\varepsilon_N I\bigr)}\,
\log\frac{1}{\varepsilon}
\right), \qquad
\kappa\!\bigl(H_t+\varepsilon_N I\bigr)
:=
\frac{\lambda_{\max}(H_t)+\varepsilon_N}{\lambda_{1}(H_t)+\varepsilon_N},
\]
iterations, and the spectral condition number $\kappa\bigl(H_t+\varepsilon_N I\bigr)$ can become arbitrarily large when $\varepsilon_N\to 0$. 

In contrast, IHSDA replaces the linear system \eqref{eq:newton-perturbed} with the homogenized eigenvalue subproblem defined by the matrix $G_t(\alpha_{t})$ from \eqref{eq:subprob}. The Lanczos method applied to this subproblem requires at most
\[
\mathcal{O}\!\left(
\sqrt{\kappa_L\!\bigl(G_t(\alpha_{t})\bigr)}\,
\log\frac{1}{\varepsilon}
\right), \quad \kappa_L\bigl(G_t(\alpha_{t})\bigr)
:=
\frac{\lambda_{\max}\bigl(G_t(\alpha_{t})\bigr)-\lambda_1\bigl(G_t(\alpha_{t})\bigr)}
{\lambda_2\bigl(G_t(\alpha_{t})\bigr)-\lambda_1\bigl(G_t(\alpha_{t})\bigr)},
\]
iterations to deliver an approximate smallest eigenpair with accuracy $\varepsilon$.
A key advantage of the homogeneous approach is that the Lanczos condition number $\kappa_L\bigl(G_t(\alpha_{t})\bigr)$ is always bounded. Indeed, for any $\alpha_{t}>0$, it follows from \cite[Theorem~2.1]{He2025HomogeneousSD} that 
\begin{equation}\label{eq:kappaL-upper}
	\kappa_L\!\bigl(G_t(\alpha_{t})\bigr)
	\leqslant
	\frac{
		2\bigl(\lambda_{\max}(H_t) - \alpha_{t} - \lambda_1\bigl(G_t(\alpha_{t})\bigr)\bigr)
	}{
		-\lambda_{\max}(H_t) + \alpha_{t}
		+ \sqrt{\bigl(\lambda_{\max}(H_t) + \alpha_{t}\bigr)^2 + \|g_t\|^2/n}
	}
	<\infty.
\end{equation}
Thus, unlike the condition number of the regularized Newton system \eqref{eq:newton-perturbed}, which can blow up as $\varepsilon_N\to 0$, the homogenized subproblem remains well-conditioned for any fixed $\alpha_{t}>0$.

To quantify the improvement, consider the degenerate case  $\lambda_1(H_t)=0$. Then \cite[Theorem~2.1]{He2025HomogeneousSD} also implies 
\begin{equation}\label{eq:kappa-ratio}
	\frac{\kappa_L\!\bigl(G_t(\alpha_{t})\bigr)}
	{\kappa\,\!\bigl(H_t+\varepsilon_N I\bigr)}
	\leqslant
	\mathcal{O}\!\left(
	\frac{\varepsilon_N}{
		\|g_t\|^2/\bigl(\lambda_{\max}(H_t)+\alpha_{t}\bigr)+\alpha_{t}
	}
	\right).
\end{equation}
The ratio \eqref{eq:kappa-ratio} directly compares the conditioning of the two inner subproblems.
In particular, if $\varepsilon_N=\alpha_{t}\to 0$ while $\|g_t\|$ stays bounded away from zero, 
then the right-hand side of \eqref{eq:kappa-ratio} converges to zero, indicating that the homogenized subproblem can be much better conditioned in this regime. 
In contrast, when $\|g_t\|\to 0$ and $\varepsilon_N$ and $\alpha_{t}$ are kept at the same scale, 
the denominator in \eqref{eq:kappa-ratio} is dominated by $\alpha_{t}$ and the ratio remains of constant order, so the two condition numbers are comparable. This behavior aligns with the practical choice $\varepsilon_N=\|g_t\|^{1/2}$ adopted in gradient norm regularized methods\cite{Doikov2024GradientRN,He2025HomogeneousSD,Mishchenko2023RegularizedNM}, yet the homogeneous formulation guarantees a bounded condition number even when the gradient is very small--a regime where Newton-type systems often become ill-conditioned.

\section{Numerical Results}
In this section, we conduct numerical experiments to demonstrate the practical performance of the proposed HSDA algorithm and its inexact variant IHSDA. We compare them with several existing methods: Gradient Descent Ascent (GDA), the IMCN algorithm~\cite{luo2022finding}, the MINIMAX-TRACE algorithm~\cite{yao2024two}, and the  IGRTR algorithm~\cite{wang2025gradient}. The experiments are performed on two representative problem classes: a low-dimensional synthetic nonconvex-strongly concave minimax problem, and an adversarial training task on the MNIST dataset. All codes are implemented in Python 3.11 and executed on a laptop equipped with an Apple M1 processor and 16 GB of memory.
\subsection{A synthetic nonconvex-strongly concave minimax problem}\label{subsec:toy}
We begin with a low-dimensional synthetic nonconvex-strongly concave minimax problem introduced in~\cite{luo2022finding}:
\begin{equation}\label{eq:toy-minimax}
    \min_{x\in\mathbb{R}^3}\max_{y\in\mathbb{R}^2}
    f(x,y)
    = w(x_3) - \frac{y_1^2}{40} + x_1y_1
      - \frac{5y_2^2}{2} + x_2y_2,
\end{equation}
where $x=[x_1,x_2,x_3]^\top$ and $y=[y_1,y_2]^\top$.  

The scalar function $w(\cdot)$ is a nonconvex, W-shaped piecewise cubic function defined by a slope parameter  $\varepsilon>0$ and a length parameter $L>1$:
\begin{equation*}\label{eq:wtoy}
w(x) =
\begin{cases}
\sqrt{\varepsilon}\bigl(x + (L+1)\sqrt{\varepsilon}\bigr)^2
 - \dfrac{1}{3}\bigl(x + (L+1)\sqrt{\varepsilon}\bigr)^3
 - c_\varepsilon,
& x \leqslant -L\sqrt{\varepsilon},\\[0.7ex]
\varepsilon x + \dfrac{\varepsilon^{3/2}}{3},
& -L\sqrt{\varepsilon} < x \leqslant -\sqrt{\varepsilon},\\[0.7ex]
-\sqrt{\varepsilon}x^2 - \dfrac{x^3}{3},
& -\sqrt{\varepsilon} < x \leqslant 0,\\[0.7ex]
-\sqrt{\varepsilon}x^2 + \dfrac{x^3}{3},
& 0 < x \leqslant \sqrt{\varepsilon},\\[0.7ex]
-\varepsilon x + \dfrac{\varepsilon^{3/2}}{3},
& \sqrt{\varepsilon} < x \leqslant L\sqrt{\varepsilon},\\[0.7ex]
\sqrt{\varepsilon}\bigl(x - (L+1)\sqrt{\varepsilon}\bigr)^2
 + \dfrac{1}{3}\bigl(x - (L+1)\sqrt{\varepsilon}\bigr)^3
 - c_\varepsilon,
& L\sqrt{\varepsilon} \leqslant x ,
\end{cases}
\end{equation*}
with $c_\varepsilon := \frac{1}{3}(3L+1)\varepsilon^{3/2}.$

In the experiment, we set $\varepsilon=0.01$, $\mu=0.05$, and $L=5$.  
Two different initial points are used:
\begin{equation*}
  (x_1,y_1)=([0.1,0.1,0.1]^{\top},[0,0]^{\top}),
  \qquad
  (x_2,y_2)=([1.0,0.1,0.1]^{\top},[0,0]^{\top}).
\end{equation*}
The first point $(x_1,y_1)$ lies near the strict saddle point $([0,0,0]^{\top},[0,0]^{\top})$ of \eqref{eq:toy-minimax}, while the second point $(x_2,y_2)$  is intentionally chosen farther from this saddle to examine the algorithms' global behavior. Figure~\ref{fig:wtoy-init1} and Figure ~\ref{fig:wtoy-init2} display the performance of the five algorithms on this problem. The horizontal axis records the iteration index $t$; the left and right vertical axes show the optimality gap $\mathcal{F}(x_t)-\mathcal{F}^{\star}$ and the gradient norm $\|\nabla \mathcal{F}(x_t)\|_2$, respectively.
\begin{figure}[htp]
	\centering 
	\includegraphics[scale=0.6]{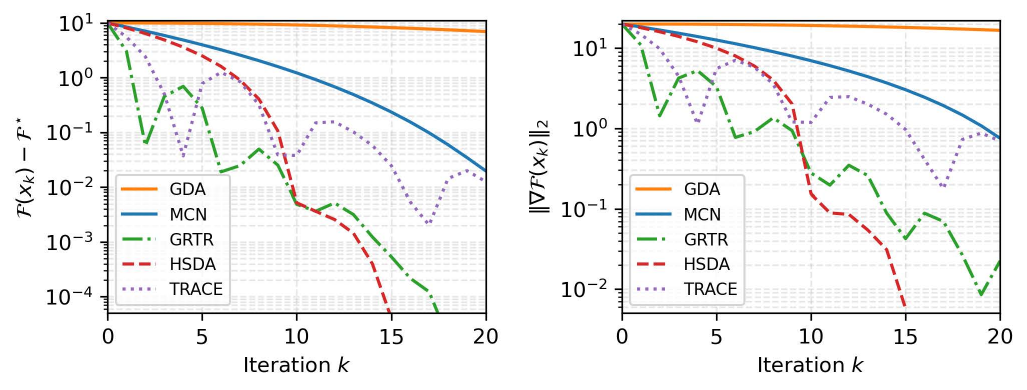}
	\caption{Numerical results of the tested algorithms on the synthetic W-shaped minimax example~\eqref{eq:toy-minimax} with initialization
		$(x_1,y_1)=([0.1,0.1,0.1]^{\top},[0,0]^{\top})$.}
	\label{fig:wtoy-init1}
\end{figure}
\begin{figure}[htp]
	\centering 
	\includegraphics[scale=0.6]{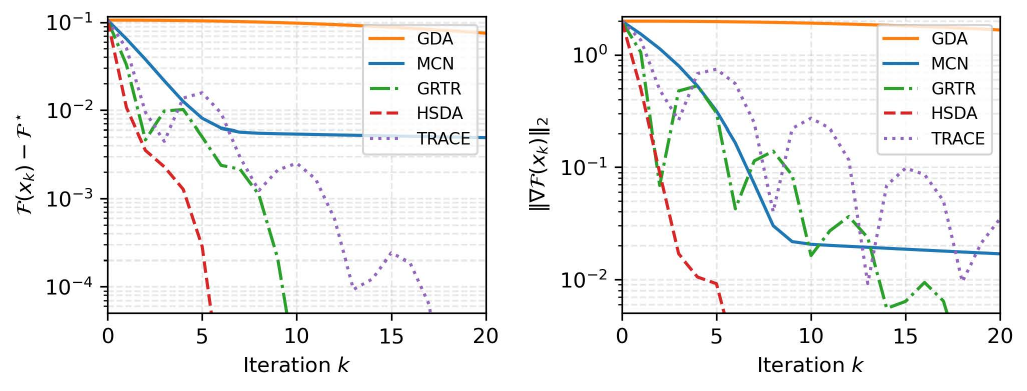}
	\caption{Numerical results of the tested algorithms on the synthetic W-shaped minimax example~\eqref{eq:toy-minimax} with initialization
		$(x_2,y_2)=([1.0,0.1,0.1]^{\top},[0,0]^{\top})$.}
	\label{fig:wtoy-init2}
\end{figure}
Since problem~\eqref{eq:toy-minimax} contains a strict saddle point and the GDA algorithm struggles to escape once trapped near it, the GDA curves in Figure~\ref{fig:wtoy-init1} and Figure ~\ref{fig:wtoy-init2}  remain almost flat, showing very little decrease in either the objective gap or the gradient norm. In contrast, the other four algorithms effectively escape the saddle region and achieve substantial progress. The proposed HSDA algorithm exhibits the fastest decrease in both the optimality gap $\mathcal{F}(x_t)-\mathcal{F}^{\star}$ and the gradient norm 
$\|\nabla \mathcal{F}(x_t)\|_2$. For both initializations, it reduces the objective gap to about 
$10^{-4}$ and the gradient norm to about  $10^{-2}$ within roughly a dozen iterations. The GRTR algorithm also converges rapidly, but its trajectories display more pronounced oscillations. The MCN algorithm reduces these quantities more slowly, yet its convergence path is comparatively smooth. With the chosen parameters, the MINIMAX-TRACE algorithm converges more slowly, and its curves are more oscillatory than those of HSDA and GRTR.

\subsection{Adversarial training on MNIST}\label{subsec:adv-mnist}
We next examine an adversarial training task studied in \cite{chen2021cubic}, whose goal is to train a classifier that remains robust against small input perturbations. Using the MNIST dataset with 50,000 training and 10,000 test samples, we solve the finite-sum minimax problem
\begin{equation}\label{eq:adv-minimax}
  \min_{x}\;\max_{y=\{y_i\}_{i=1}^n}
  \frac{1}{n}\sum_{i=1}^n
  \Bigl[
    \ell\bigl(h_x(y_i),b_i\bigr)
    - \lambda \,\|y_i-a_i\|_2^2
  \Bigr],
\end{equation}
where $x$ collects the parameters of a convolutional neural network
$h_x(\cdot)$, $(a_i,b_i)$ denotes the $i$th image-label pair, and
$y_i\in\mathbb{R}^{784}$ is an adversarial version of $a_i$. Following \cite{chen2021cubic},
we take $\ell$ to be the cross-entropy loss and set $\lambda=2$.

The network $h_x$ is a simple convolutional architecture: a single convolutional block (one input channel, one output channel, kernel size $3$, stride $4$, padding $1$, followed by a sigmoid activation) produces a  $7\times7$ feature map; this map is flattened to a $49$-dimensional vector and passed through a linear layer with $10$ outputs. All network parameters are stacked into a vector $x\in\mathbb{R}^{d_x}$; each adversarial variable $y_i$  is initialized at the original image  $a_i$.

We apply IHSDA together with GDA, IMCN~\cite{luo2022finding},
IGRTR~\cite{wang2025gradient}, and ILMNegCur~\cite{wang2025gradient} to
\eqref{eq:adv-minimax}. All methods are run in mini‑batch mode with batch size $64$.
For the inner maximization over  $y$, every algorithm approximates the maximizer by a limited number of (possibly accelerated) gradient-ascent steps. For IHSDA we set the strong-concavity parameter in the $y$-direction to $\mu=1$, the Lipschitz constant of  $\nabla_y f$ to $\ell=10$, and the Hessian Lipschitz constant of the value function to $L_2=0.2$. The homogeneous second‑order subproblem in each outer iteration is solved approximately by a Lanczos procedure limited to at most $80$ iterations. IMCN is implemented according to the description in \cite{luo2022finding}, while IGRTR and ILMNegCur follow the specifications in 
\cite{wang2025gradient}.

Figure~\ref{fig:adv-mnist} presents the results. Panel~(a)  plots test accuracy against wall-clock time, and panel~(b) shows the approximate objective function value of~\eqref{eq:adv-minimax} versus the outer iteration index. On this problem, the four second-order algorithms achieve higher test accuracies than GDA within the same time budget. IHSDA typically reaches a test accuracy around $80\%$ and attains the lowest objective values among the compared methods. IMCN, IGRTR, and ILMNegCur are competitive and follow closely. GDA improves more gradually and remains below about $70\%$ accuracy over the plotted range. Overall, the curves indicate that exploiting second-order information through the HSDA framework is beneficial for this adversarial training task, and that the resulting IHSDA method performs on par with existing second-order schemes.
\begin{figure}[t]
  \centering
  \includegraphics[width=\textwidth]{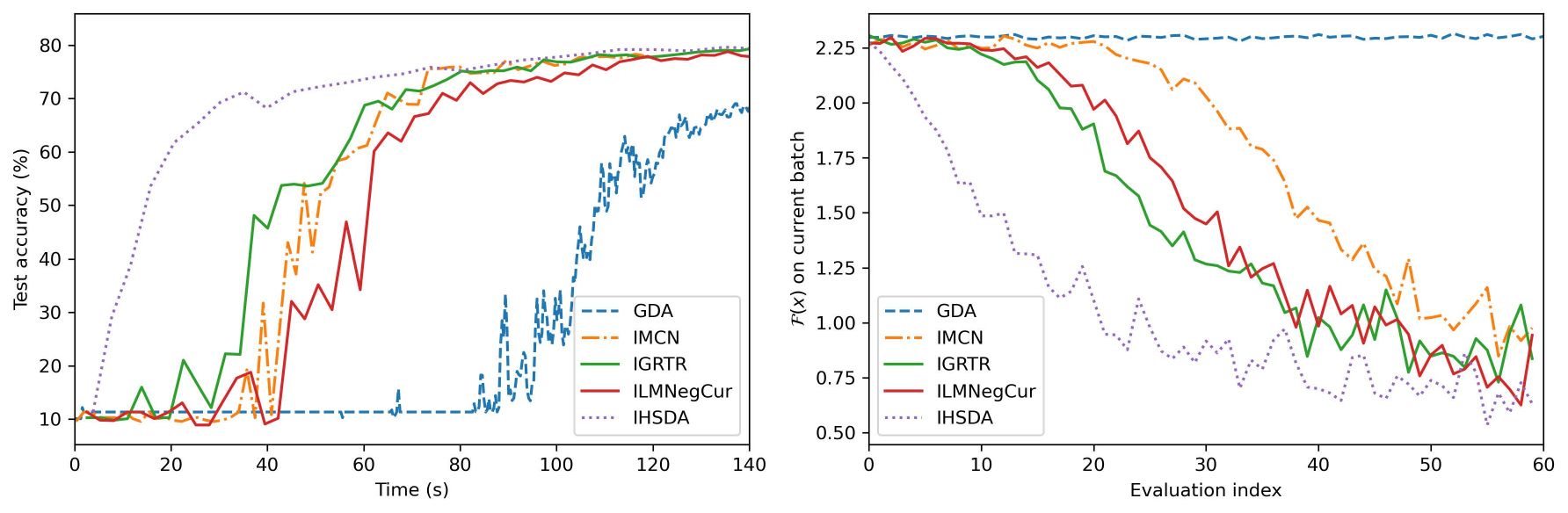}
  \caption{Numerical results of the tested algorithms for solving~\eqref{eq:adv-minimax}.}
  \label{fig:adv-mnist}
\end{figure}
\section{Conclusions}
In this paper, we have introduced a Homogeneous Second-Order Descent Ascent (HSDA) algorithm and its inexact variant (IHSDA) for solving nonconvex-strongly concave minimax problems. The algorithms leverage a homogenized eigenvalue subproblem to compute a search direction that ensures sufficient descent even when the Hessian of the value function is nearly positive semidefinite.

We prove that both HSDA and IHSDA find an $\mathcal{O}(\varepsilon,\sqrt{\varepsilon})$-second-order
stationary point within at most $\tilde{\mathcal{O}}(\varepsilon^{-3/2})$ outer iterations, matching the best known iteration complexity for existing second-order methods in this setting. For the practical IHSDA variant, which solves the subproblem approximately via a Lanczos procedure, we further establish a high-probability bound of $\tilde{\mathcal{O}}(\varepsilon^{-7/4})$ for the total number of Hessian-vector products.

The numerical experiments on synthetic minimax problems and adversarial training tasks confirm the efficiency and robustness of the proposed methods. A natural and promising direction for future work is to extend the homogeneous second-order framework beyond the nonconvex-strongly concave setting, e.g., to more general minimax structures that appear in modern machine learning applications.
 \section*{Data Availability}
No datasets were generated or analysed during the current study.
 \section*{Declarations}
 The authors declare that they have no conflict of interest.


\end{document}